\documentclass[11pt]{amsart}

\usepackage[margin=1.20in]{geometry}

\usepackage[T1]{fontenc}
\usepackage{mathtools}
\usepackage{amssymb}
\usepackage{newtxtext,newtxmath}
\usepackage{enumitem}
\usepackage{xcolor}
\usepackage[english]{babel}

\setlist[itemize]{
  topsep=4pt,
  itemsep=2pt,
  parsep=0pt,
  partopsep=0pt
}

\setlist[enumerate]{
  topsep=4pt,
  itemsep=2pt,
  parsep=0pt,
  partopsep=0pt
}

\allowdisplaybreaks[2]

\definecolor{linkblue}{HTML}{1F4E79}

\usepackage[
  colorlinks=true,
  linkcolor=linkblue,
  citecolor=linkblue,
  urlcolor=linkblue,
  pagebackref=true,
  pdfdisplaydoctitle=true
]{hyperref}

\renewcommand*{\backref}[1]{}

\renewcommand*{\backrefalt}[4]{%
  \ifcase #1
    Not cited.%
  \or
    Cited on page~#2.%
  \else
    Cited on pages~#2.%
  \fi
}

\usepackage{bookmark}

\usepackage[
  nameinlink,
  capitalise,
  noabbrev
]{cleveref}

\hypersetup{
  pdftitle={
    Polynomial-Time Lattice-Point Counting without Barvinok Decomposition
  },
  pdfauthor={
    Guoce Xin and Zihao Zhang
  },
  pdfsubject={
    Polynomial-time lattice-point counting in fixed dimension
    without Barvinok decomposition
  },
  pdfkeywords={
    vector partition, Sylvester denumerant, constant term,
    partial fraction, Minkowski theorem, Smith normal form,
    fixed dimension
  }
}

\numberwithin{equation}{section}

\theoremstyle{plain}

\newtheorem{theorem}{Theorem}[section]
\newtheorem{lemma}[theorem]{Lemma}
\newtheorem{proposition}[theorem]{Proposition}

\theoremstyle{definition}

\newtheorem{definition}[theorem]{Definition}
\newtheorem{algorithm}[theorem]{Algorithm}

\theoremstyle{remark}

\newtheorem{remark}[theorem]{Remark}

\crefname{theorem}{theorem}{theorems}
\crefname{lemma}{lemma}{lemmas}
\crefname{proposition}{proposition}{propositions}
\crefname{corollary}{corollary}{corollaries}
\crefname{definition}{definition}{definitions}
\crefname{algorithm}{algorithm}{algorithms}
\crefname{example}{example}{examples}
\crefname{remark}{remark}{remarks}

\DeclareMathOperator{\diag}{diag}

\newcommand{\ZZ}{\mathbb Z}
\newcommand{\QQ}{\mathbb Q}

\newcommand{\y}{\mathbf y}

\newcommand{\nv}{\nu}

\newcommand{\cL}{\mathcal L}

\newcommand{\K}{\mathbb K}
\newcommand{\ind}{\operatorname{ind}}
\newcommand{\abs}[1]{\lvert#1\rvert}
\newcommand{\norm}[1]{\lVert#1\rVert}
\newcommand{\angles}[1]{\langle#1\rangle}
\newcommand{\V}{\mathsf V}

\title[Lattice-Point Counting without Barvinok]
{Polynomial-Time Lattice-Point Counting without Barvinok Decomposition}

\author{Guoce Xin\,\textsuperscript{1}}
\author{Zihao Zhang\,\textsuperscript{2}}



\makeatletter

\newcommand{\firstpageauthorinfo}{%
  \par\vspace{0.65em}
  \begingroup
  \centering
  \begin{minipage}{0.94\textwidth}
    \centering\small

    \textsuperscript{1}%
    School of Mathematical Sciences,
    Capital Normal University, 
     Beijing 100048, P.R.~China\\[-1pt]
    \href{mailto:guoce_xin@163.com}
      {\nolinkurl{guoce_xin@163.com}}

    \vspace{0.45em}

    \textsuperscript{2}%
    School of Mathematics and Statistics,
    Beijing Institute of Technology,
    Beijing 102400, P.R.~China\\[-1pt]
    \href{mailto:zihao-zhang@foxmail.com}
      {\nolinkurl{zihao-zhang@foxmail.com}}

  \end{minipage}\par
  \endgroup
  \vspace{0.35em}
}

\let\original@setauthors\@setauthors

\renewcommand{\@setauthors}{%
  \original@setauthors
  \firstpageauthorinfo
}

\makeatother

\begin{document}

\begin{abstract}

By using constant term manipulations, we present the first polynomial-time algorithm for lattice-point counting in fixed dimension that does not rely on Barvinok's unimodular decomposition. Our method starts from partial fraction decompositions and root-of-unity formulas for the lattice-point generating function of a rational simplicial cone. By means of a residue-lattice argument based on Minkowski's theorem, we construct a short multiplier that enables an exact recursive reduction. Smith normal form is then used to normalize the resulting child terms and restore the recursive structure. These transformations are justified by algebraic identities involving roots of binomial equations, while the algorithm itself works with integer exponent data and requires no explicit root computations. 
More precisely, for a rational simplicial cone \(\mathcal K\) in fixed dimension \(d\), with dual cone \(\mathcal K^*\), the algorithm achieves recursion depth \(O_d(1+\log\log(2+\ind(\mathcal K^*)))\) and produces a signed sum of at most \((1+\log\ind(\mathcal K^*))^{O_d(1)}\) unimodular cone generating functions. The framework uniformly handles numerators that are Laurent polynomials, not merely monomials, thereby giving a polynomial-time algorithm for MacMahon's partition analysis when the dimension is fixed.

\end{abstract}

\maketitle
\section{Introduction}
  Let
\(A\in\ZZ^{m\times n}\) have full row rank and let
\(\mathbf b\in\ZZ^m\).  Define the rational polyhedron
\[
 \mathcal P(A,\mathbf b)
 :=\{\mathbf x\in\mathbb R_{\geq0}^n:A\mathbf x=\mathbf b\}.
\]
For \(\y=(y_1,\ldots,y_n)\) and
\(\boldsymbol\alpha=(\alpha_1,\ldots,\alpha_n)\in\ZZ^n\), write
\(
 \y^{\boldsymbol\alpha}:=\prod_{i=1}^n y_i^{\alpha_i},
\)
and define the integer-point transform for any set \(S\subseteq\mathbb R^n\)
\begin{equation}\label{eq:intro-vector-partition}
 \sigma_{S}(\y):
 =\sum_{\mathbf x\in S\cap\ZZ^n}
 \y^{\mathbf x}.
\end{equation}
When \( \mathcal P:=\mathcal P(A,\mathbf b)\) is bounded, this transform is a polynomial
and, with \(\mathbf1:=(1,\ldots,1)\),
\[
 \#\bigl(\mathcal P \cap\ZZ^n\bigr)
 =\sigma_{\mathcal P }(\mathbf1).
\]
We call \(\sigma_{\mathcal P(A,\mathbf b)}(\y)\)
the generating function of \(\mathcal P(A,\mathbf b)\).

A basic geometric building block for these transforms is a full-dimensional rational simplicial cone \(\mathcal K\subseteq\mathbb R^d\).
For
\(W=(\mathbf r_1,\ldots,\mathbf r_p)\in\mathbb Q^{n\times p}\), write
\[
 \mathcal K(W)
 :=\operatorname{cone}(\mathbf r_1,\ldots,\mathbf r_p)
 =\{W\boldsymbol k:
   \boldsymbol k \in\mathbb R_{\geq0}^p\}.
\]
When \(W\) is understood, we abbreviate \(\mathcal K(W)\) to
\(\mathcal K\); every cone whose integer-point transform is used is
assumed to be pointed.  For \(\mathbf v\in\mathbb Q^n\), set
\[
 \mathcal K^{\mathbf v}:=\mathbf v+\mathcal K,
 \qquad
 \sigma_{\mathcal K^{\mathbf v}}(\y)
 :=\sum_{\mathbf x\in\mathcal K^{\mathbf v}\cap\mathbb Z^n}
 \y^{\mathbf x}.
\]
If \(\mathbf r_1,\ldots,\mathbf r_p\) are linearly independent, the cone is called simplicial.
When \(p=n\), it is said to be full-dimensional, and we define
\[
\ind(\mathcal K) = \bigl|\det(\mathbf r_1,\ldots,\mathbf r_p)\bigr|.
\]
In this definition, the \(\mathbf r_i\) are taken to be the
primitive integer generators of the rays.

Barvinok's algorithm
\cite{Barvinok1994,BarvinokPommersheim1999}
represents a landmark advance in lattice point enumeration.
Geometrically,
the algorithm first uses Brion's theorem to reduce the computation to
the tangent cones at the vertices. It then triangulates these cones
\cite{Lasserre1983,Lawrence1991}. That is,
\begin{equation} \label{eq: simplicalcone}
  \sigma_{\mathcal P}(\y) = \sum_{i}  \pm \sigma_{\mathcal K_i^{\mathbf v_i}}(\y),
\end{equation}
where \(\mathcal K_i\) is a simplicial cone.

For fixed dimension \(d\), Barvinok shows that there exists a polynomial-time algorithm decomposing the generating function of \(\mathcal K_i\) into a signed sum of generating functions of unimodular cones \(\mathcal U_\ell\) (\(\ind(\mathcal U_\ell)=1\)), which can be written directly as
\begin{equation}\label{eq:barvinok-output}
\begin{aligned}
 \sigma_{\mathcal K^{\mathbf v}}(\y)
 &=\sum_{\ell=1}^{N}\varepsilon_\ell
   \sigma_{\mathcal U_\ell}(\y)
 =\sum_{\ell=1}^{N}\varepsilon_\ell
   \frac{\y^{\mathbf a_\ell}}
   {\displaystyle\prod_{j=1}^d
    \left(1-\y^{\mathbf b_{\ell j}}\right)},\quad \varepsilon_\ell\in\{-1,1\},
 \quad \mathbf a_\ell,\mathbf b_{\ell j}\in\mathbb Z^d.
\end{aligned}
\end{equation}
Barvinok showed that \(N\) is bounded by
\[
\left(1+\log(\ind(\mathcal K)) \right)^{O_d(1)}.
\]

Laurent polynomial numerators arise naturally in MacMahon's
partition analysis (MPA). They are readily accommodated by
constant-term and root-averaging methods, with relevant
algorithms such as
\texttt{Omega} \cite{andrews2001macmahonOmegaalgorithm2}, \texttt{GenOmega} \cite{Han2003},
\texttt{Ell} \cite{Xin2004},  \texttt{CTEuclid} \cite{Xin2015}, \texttt{Polyhedral Omega} \cite{2017Polyhedral}, and so on. 

The Barvinok algorithm does not perform as well as certain non-polynomial-time MPA algorithms on many combinatorial problems, such as magic square computation 
\cite{Xin2015} and the Sdd problem \cite{GarsiaMusikerWallachXin2009}. Yet, to our knowledge, no polynomial-time algorithm independent of Barvinok's approach has been discovered so far.
This observation motivates the present work: we develop a
polynomial-time algorithm in fixed dimension for rational shifted
simplicial cones using the normalized root averaging operator,
without Barvinok's unimodular decomposition.
The framework also accommodates explicitly given Laurent
polynomial numerators.

Barvinok's method has also led to mature software implementations.
\texttt{LattE} provided the
first implementation of Barvinok's unimodular-cone decomposition and
supports exact lattice-point counting, Ehrhart computations, and
integration over rational polytopes
\cite{DeLoeraHemmeckeTauzerYoshida2004}.  The \texttt{barvinok} library extends this framework to parametric and nonparametric polytopes and implements the projection algorithm due to Barvinok and Woods \cite{BarvinokWoods2003,koppe2008implementation}.

Note that when \(\ind(\mathcal K)\) is treated as a fixed parameter, Barvinok's fixed-dimensional decomposition is no longer the most efficient option. For certain structured families of \(\Delta\)-modular polyhedra, tailored algorithms admit polynomial-time complexity even with varying dimensions, provided that the underlying structural parameters remain fixed; see the related work in \cite{GribanovZolotykh2022} and the references therein. These approaches, however, rely on special modular constraints and do not apply to general rational simplicial cones.

For a bounded polytope \(\mathcal P(A,\mathbf b)\), the polynomial \(\sigma_{\mathcal P}(\y)\) is well-defined at the specialization \(\y=\mathbf{1}\), while the individual rational functions appearing in short representations such as \eqref{eq:barvinok-output} typically exhibit singularities at this point. The desired specialization is therefore rigorously implemented via the well-defined substitution \(y_i = \exp(\beta_i s)\): one extracts the constant term in \(s\) from each summand and accumulates the resulting contributions. In practice, the parameters \(\beta_i\) are usually chosen randomly, though deterministic polynomial-time selections are also available \cite{DakhnoGribanovKasianov2024}. Once a short representation consisting of \(N\) summands is constructed, the approach proposed in \cite{XinZhangZhang2023} realizes this specialization with only
\[
O\!\left(Nd\log^2 d\right)
\]
arithmetic operations over a finite field \(\mathbb F_p\). A sufficiently large set of modular evaluations then recovers the exact lattice-point count via the Chinese remainder theorem.

In combinatorics,
unlike triangulation and Brion's theorem, an alternative algebraic decomposition algorithm to obtain the simplicial cone decomposition \eqref{eq: simplicalcone} is provided by the \texttt{SimpCone[S]} \cite{XinXuZhang2025}.   Related applications of this framework include
polytope-volume computation \cite{XinXuZhangZhang2024} and the Ehrhart
series of magic squares of order seven and eight \cite{QiuXinZhang2026}.

We follow the definitions in \cite{XinXuZhang2025}; the generating function of a simplicial cone can be expressed via the normalized root averaging operator, defined as follows. Let \(\mathbb L\supseteq\mathbb Q(\y)\) be a field of characteristic zero, and let \(\overline{\mathbb L}\) denote its fixed algebraic closure. For any \(u\in\mathbb L^*\), integer \(s\geq 1\), and rational function \(H\in\mathbb L(\lambda)\) that is regular at every root of \(u\lambda^s=1\), define
\[
 \Phi_{u,s}^{\lambda}H(\lambda)
 :=\frac1s
   \sum_{\substack{\omega\in\overline{\mathbb L}^{\,*}\\
                    u\omega^s=1}}
   H(\omega).
\]
For a given rational simplicial cone \(\mathcal K\),
\(\sigma_\mathcal K\)
has
the form
\begin{equation}\label{eq:intro-successive-ct}
\begin{aligned}
 &\Phi_{u_1,s_1}^{\lambda_1}\cdots
  \Phi_{u_h,s_h}^{\lambda_h}F(\Lambda;\y)\\
\end{aligned},
\end{equation}
where
$$F(\Lambda;\y)
  :=\frac{c\,u_0\Lambda^{\mathbf c_0}}
  {\prod_{\ell=1}^{r}
   \left(1-u_{h+\ell}\Lambda^{\mathbf c_\ell}\right)}, \qquad \Lambda=(\lambda_1,\ldots,\lambda_h).$$
 More generally, for a rational polyhedron \(\mathcal P(A,\mathbf b)\), \texttt{SimpCone[S]} efficiently decomposes its generating function into finitely many summands of the form \eqref{eq:intro-successive-ct}.
For detailed formulas see Section~\ref{sec:general-reduction}.

In \cite{XinZhangZhang2024}, the \texttt{DecDenu} method evaluates \eqref{eq:intro-successive-ct} one variable at a time. It relies on the core identity
\[
\Phi_{u,s}^{\lambda}H(\lambda)=\Phi_{u^{1/k},s}^{\lambda}H(\lambda^k),
\]
which holds under the condition \(\gcd(s,k)=1\). Nevertheless, existing recursion-based analysis fails to deliver a fixed-dimensional polynomial bound, even though its practical performance is competitive with \texttt{LattE} on certain families of problems.

The present work extends this reduction to non-coprime multipliers
and establishes a polynomial-time bound in fixed dimension for computing \eqref{eq:intro-successive-ct} globally.
Its main ingredients are as follows.
\begin{enumerate}[label=\textup{(\roman*)},leftmargin=*]
\item A residue-lattice argument based on Minkowski's theorem
\cite{Cassels1997} produces a short multiplier; in fixed dimension, such a
multiplier can be found deterministically by Lenstra's algorithm
\cite{Lenstra1983}.  Using only gcd computations and the Chinese remainder
theorem, we then construct a unit lift of the reduced multiplier.  This
gives an exact split of the selected root average into two nested averages,
uses only integral exponents, and requires no integer factorization.

\item The split couples the new root equations, so the resulting operators
are not independent scalar averages.  We encode each pole-exchange child
as a joint root average, put its exponent matrix into Smith normal form,
and return the diagonalized term to the recursion.

\item Two structural invariants make the recursion well defined.  The
generation condition supplies focused primitivity and the bound
\(h\leq r\leq d\), where \(h\) is the number of carried equations and
\(r\) is the number of remaining denominator factors.  Full-column
independence of the exponent matrix and saturation of its column lattice
give the regularity and pairwise coprimality required by pole exchange.
\end{enumerate}

For a directly supplied fixed-dimensional rational simplicial cone \(\mathcal K\), denote by \(\mathcal K^*\) the dual cone of \(\mathcal K\).
Theorem~\ref{thm:complexity} shows that Algorithm \ref{alg:jointct-kernel} has recursion depth
\[
 O_d\bigl(1+\log\log(2+\ind(\mathcal K^*))\bigr)
\]
and expresses the cone transform as a signed sum of at most
\[
 (1+\log \ind(\mathcal K^*))^{O_d(1)}
\]
generating functions of possibly shifted unimodular simplicial cones.  If
a nonterminal state has determinant \(D\), then every child satisfies
\[
 D_{\mathrm{child}}
 \leq
 \min\!\left\{D^{1-1/d^2},\left\lfloor\frac D2\right\rfloor\right\}.
\]
Here \(d\) is the dimension of the directly supplied cone, whereas \(n\)
is the ambient dimension of a general-system input.
For fixed \(n-m\), the decomposition \eqref{eq: simplicalcone}
can be computed in polynomial time by Brion's theorem and
triangulation \cite[Sections~1--3]{BarvinokPommersheim1999}.
For each cone, Smith normal form gives integer affine coordinates
in its affine hull in polynomial time, or certifies that the hull
contains no lattice point. In the latter case the contribution is zero.
Zero-dimensional terms are evaluated directly.
We apply Algorithm~\ref{alg:jointct-kernel} to each remaining cone
in these coordinates, using Proposition~\ref{prop:simplicial-cone-ct}
for initialization, and map its generating function back by the
corresponding monomial substitution and multiplication by the shift
monomial.
The number of cones and the encoding length of each cone are
polynomially bounded, so summing their generating functions gives
\(\sigma_{\mathcal P(A,\mathbf b)}(\y)\) in polynomial time.
For a Laurent polynomial numerator, we compute each monomial's
contribution separately and add the results.
The numerator is given explicitly, with the binary lengths of
all coefficients and exponents included in the input size.

The paper is organized as follows.  Section~\ref{sec:general-reduction}
develops the two constant-term input routes and identifies their common
selected-pole form.  Section~\ref{sec:local-calculus} reviews the local
pole calculus and the coprime \texttt{DecDenu} step, and
Section~\ref{sec:noncoprime} proves the exact non-coprime split.
Section~\ref{sec:recursive-terms} introduces joint root averages and Smith
normalization.  Section~\ref{sec:jointct-algorithm} assembles the queue
algorithm, while Section~\ref{sec:jointct-invariants} establishes its
structural invariants and well-definedness.  The fixed-dimensional depth
and output bounds are proved in Section~\ref{sec:complexity}, followed by
concluding remarks in Section~\ref{sec:conclusion}.

\section{Simplicial-cone constant terms and the common operator problem}
\label{sec:general-reduction}

We now make the introductory formula precise.  After fixing the minimal
notation for selected-pole averages, we begin with the direct
simplicial-cone construction.  We then recall the
\texttt{SimpCone[S]}--Smith reduction for a general system and show that
the two routes produce the same class of diagonal operator inputs.  All
cone identities in this section are identities of rational generating
functions; when series expansions are required, they are taken in a
common field of iterated Laurent series.

\subsection{Basic notation and selected-pole averages}

For \(A\in\mathbb Z^{m\times n}\), write \(A_j:=A\mathbf e_j\) for its
\(j\)th column and set \([n]:=\{1,\ldots,n\}\).  If
\(J=\{j_1<\cdots<j_p\}\subseteq[n]\), then
\[
 A_J:=\bigl(A\mathbf e_{j_1},\ldots,A\mathbf e_{j_p}\bigr).
\]
For \(\boldsymbol\alpha=(\alpha_1,\ldots,\alpha_n)\in\mathbb Z^n\), we
use the monomial notation
\[
 \y^{\boldsymbol\alpha}:=\prod_{i=1}^n y_i^{\alpha_i}.
\]
The variables \(\lambda_i\) and \(\nv_i\) are formal binding variables;
actual roots substituted for them are denoted by \(\omega_i\).  The marker
field is \(\K:=\QQ(\y)\).  The symbol \(\mathbb L\) denotes a
characteristic-zero coefficient field containing \(\K\); in a local
calculation it may be enlarged by adjoining the binding variables other
than the one currently in focus.

We write \(\Phi\) for a one-variable root average and \(\Psi\) for a
joint root average.  In a normalized recursive state, every \(u_\ell\)
is a Laurent monomial in the marker variables alone.  In the local one-variable calculus, a
coefficient may additionally involve the other binding variables, but it
never involves the focused variable.

\begin{definition}\label{def:phi}
Let \(u\in\mathbb L^*\) and \(a\in\mathbb Z_{\geq1}\), and write
\begin{equation*}
 \mathcal R_a(u)
 =\{\omega\in\overline{\mathbb L}^{\,*}:u\omega^a=1\}.
\end{equation*}
If \(H(\lambda)\) is defined at every point of \(\mathcal R_a(u)\), set
\begin{equation*}
 \Phi_{u,a}^{\lambda}H(\lambda)
 =\frac1a\sum_{\omega\in\mathcal R_a(u)}H(\omega).
\end{equation*}
Thus, if \(\omega_0\in\mathcal R_a(u)\), then
\[
 \mathcal R_a(u)
 =\omega_0\{\zeta\in\overline{\mathbb L}^{\,*}:\zeta^a=1\}.
\]
The value is independent of the chosen splitting field and belongs to
\(\mathbb L\) whenever
\(H\in\mathbb L(\lambda)\).
\end{definition}

We record this root average by the following underlined selected-pole
notation:
\begin{equation*}
 \left(
  \underset{\lambda}{\operatorname{CT}}\,
  \frac{1}{\underline{1-u\lambda^a}}
 \right)H(\lambda)
 :=
 \Phi_{u,a}^{\lambda}H(\lambda).
\end{equation*}
Such notation also arises naturally from partial-fraction decomposition, as described in Section~\ref{sec:local-calculus}.
  It does not
denote the ordinary constant term of the complete quotient, which may
receive contributions from several poles and depends on the chosen
iterated Laurent-series expansion \cite{Xin2004}.  Therefore, we may successively define
\begin{equation*}
\begin{aligned}
&\left(
 \underset{\lambda_h}{\operatorname{CT}}\,
 \frac{1}{\underline{1-\lambda_h^{s_h}u_h}}
\right)
\circ\cdots\circ
\left(
 \underset{\lambda_1}{\operatorname{CT}}\,
 \frac{1}{\underline{1-\lambda_1^{s_1}u_1}}
\right)
H(\Lambda).
\end{aligned}
\end{equation*}
Note that these operators commute with one another, so the result is independent of the order of composition.

\subsection{The constant-term model of a simplicial cone}

We begin with a rational simplicial cone supplied directly.  This is the
basic constant-term input from which the joint recursion is developed.

It should be noted that the generating-function formula for the simplicial cone below has already been proven via the  \texttt{simpcone} algorithm in \cite[Section 5.3]{XinXuZhang2025}, but we provide an independent proof here.

\begin{proposition}
\label{prop:simplicial-cone-ct}
Let \(\mathcal K\subset\mathbb R^d\) be a full-dimensional rational
simplicial cone, and let \(\mathbf v\in\mathbb Q^d\).
Let \(B\in\mathbb Z^{d\times d}\) have as its columns the primitive
integer generators of the rays of \(\mathcal K^*\).
Choose \(P,Q\in\operatorname{GL}_d(\mathbb Z)\) such that
\[
 PB^{\mathsf T}Q
 =S:=\diag(s_1,\ldots,s_d),
 \qquad
 1\leq s_1\mid\cdots\mid s_d.
\]
Define \(\mathbf b:=\lceil B^{\mathsf T}\mathbf v\rceil\), with the
ceiling taken componentwise.
Writing \(\y=(y_1,\ldots,y_d)\) and
\(\Lambda=(\lambda_1,\ldots,\lambda_d)\), we have
\begin{equation}\label{eq:simplicial-cone-ct}
\begin{aligned}
\sigma_{\mathcal K^{\mathbf v}}(\y)
={}&
\left(
 \underset{\lambda_d}{\operatorname{CT}}\,
 \frac{1}
 {\underline{1-\lambda_d^{s_d}
                 \y^{Q\mathbf e_d}}}
\right)
\circ\cdots\qquad{}\circ
\left(
 \underset{\lambda_1}{\operatorname{CT}}\,
 \frac{1}
 {\underline{1-\lambda_1^{s_1}
                 \y^{Q\mathbf e_1}}}
\right)
\frac{\Lambda^{-P\mathbf b}}
{\displaystyle\prod_{j=1}^d
 \left(1-\Lambda^{-P\mathbf e_j}\right)}.
\end{aligned}
\end{equation}
The cofactor on the right-hand side is regular on the common root set,
so every displayed selected-pole operator is well defined.

\end{proposition}

\begin{proof}
Put \(\mathbf q_i:=Q\mathbf e_i\) and
\(\mathbf p_j:=P\mathbf e_j\).  The roots-of-unity filter gives
\[
 \Phi_{u,s}^{\lambda}(\lambda^{-m})
 =
 \begin{cases}
  u^{m/s},&s\mid m,\\
  0,&s\nmid m,
 \end{cases}
 \qquad m\in\mathbb Z.
\]
Let
\(
 \Phi:=
 \Phi_{\y^{\mathbf q_d},s_d}^{\lambda_d}
 \circ\cdots\circ
 \Phi_{\y^{\mathbf q_1},s_1}^{\lambda_1}.
\)
Since
\[
 \frac{\Lambda^{-P\mathbf b}}{\prod_{j=1}^d(1-\Lambda^{-\mathbf p_j})}
 =
 \sum_{\mathbf k\in\mathbb Z_{\geq0}^d}
 \Lambda^{-P(\mathbf b+\mathbf k)},
\]
the roots-of-unity filter gives
\[
 \begin{aligned}
 \Phi\left(
 \frac{\Lambda^{-P\mathbf b}}{\prod_{j=1}^d(1-\Lambda^{-\mathbf p_j})}
 \right)
 &=
 \sum_{\mathbf k\in\mathbb Z_{\geq0}^d}
 \Phi(\Lambda^{-P(\mathbf b+\mathbf k)})
 =
 \sum_{\substack{\mathbf k\in\mathbb Z_{\geq0}^d\\
                  P(\mathbf b+\mathbf k)\in S\mathbb Z^d}}
 \y^{QS^{-1}P(\mathbf b+\mathbf k)}.
 \end{aligned}
\]

On the other hand, \(PB^{\mathsf T}Q=S\) implies
\[
 QS^{-1}P=(B^{\mathsf T})^{-1}.
\]
Since \(Q\) is unimodular, \(P(\mathbf b+\mathbf k)\in S\mathbb Z^d\)
exactly when \(\mathbf x:=QS^{-1}P(\mathbf b+\mathbf k)\) is integral.
For such \(\mathbf x\), we have
\(\mathbf k=B^{\mathsf T}\mathbf x-\mathbf b\), and hence
\[
 \mathbf k\geq0
 \quad\Longleftrightarrow\quad
 B^{\mathsf T}\mathbf x\geq\mathbf b
 \quad\Longleftrightarrow\quad
 \mathbf x\in\mathcal K^{\mathbf v},
\]
where the last equivalence follows from
\(\mathbf b=\lceil B^{\mathsf T}\mathbf v\rceil\) and the integrality of
\(B^{\mathsf T}\mathbf x\).
Thus \(\mathbf k\mapsto\mathbf x\) is a bijection onto
\(\mathcal K^{\mathbf v}\cap\mathbb Z^d\), and the preceding sum is
\(\sigma_{\mathcal K^{\mathbf v}}(\y)\).

Finally, let \(\boldsymbol\omega\) be a common root and suppose
\(\boldsymbol\omega^{-\mathbf p_j}=1\).  Set \(N=s_d\) and
\(\mathbf w:=NS^{-1}\mathbf p_j\in\mathbb Z^d\).  Using
\(\omega_i^{s_i}=\y^{-\mathbf q_i}\), we get
\[
 1=(\boldsymbol\omega^{-\mathbf p_j})^N
   =\boldsymbol\omega^{-S\mathbf w}
   =\y^{Q\mathbf w}.
\]
But \(\mathbf p_j\neq0\), so \(Q\mathbf w\neq0\), contradicting the
algebraic independence of the marker variables.  Hence the cofactor is
regular on the common root set, and all selected-pole operators are
well defined.
\end{proof}

We next return to the notation \(\mathcal P(A,\mathbf b)\) and
\(\sigma_{ \mathcal P( A,\mathbf b)}(\y)\) from \eqref{eq:intro-vector-partition}.
We use the following consequence of the \texttt{SimpCone[S]}
construction and its Smith-normal-form postprocessing; see
Theorem~14, Algorithm~18, and Section~5.2, especially
Equation~(5.7), of \cite{XinXuZhang2025}.

\begin{proposition}[SimpCone--Smith reduction to successive constant terms]
\label{prop:simpcone-smith}
Let \(A\in\mathbb Z^{m\times n}\) have rank \(m\), and let
\(\mathbf b\in\mathbb Z^m\).  The \texttt{SimpCone[S]} algorithm
produces a finite collection
\[
\mathcal O_m
=
\left\{
 (\varepsilon_\iota,J_\iota,\mathcal K_\iota):
 \iota\in\mathcal I
\right\},
\]
where \(\varepsilon_\iota\in \ZZ\),
\(J_\iota\subseteq[n]\) has cardinality \(m\), and
\(\mathcal K_\iota\) is a rational shifted simplicial cone of dimension
\(n-m\).  As an
identity of rational generating functions,
\begin{equation*}
\sigma_{ \mathcal P( A,\mathbf b)}(\y)
=
\sum_{\iota\in\mathcal I}
\varepsilon_\iota\,
\sigma_{\mathcal K_\iota}(\y).
\end{equation*}

Fix \(\iota\in\mathcal I\) and write
\[
 J=J_\iota=\{j_1<\cdots<j_m\},
 \qquad
 J^c=\{k_1<\cdots<k_{n-m}\}.
\]
When \(n=m\), \(J^c\) and every tuple or product indexed by it are
understood to be empty.
The pivot block \(A_J\) is
nonsingular.  Choose \(P,Q\in\operatorname{GL}_m(\mathbb Z)\) such that
\[
 PA_JQ=S:=\diag(s_1,\ldots,s_m),
 \qquad
 1\leq s_1\mid\cdots\mid s_m,
\]
and put
\[
 \Lambda=(\lambda_1,\ldots,\lambda_m),
 \qquad
 \mathbf q_i=\sum_{k=1}^m Q_{ki}\mathbf e_{j_k}
 \quad(1\leq i\leq m).
\]
Then the Smith postprocessing gives
\begin{equation}\label{eq:successive-decd}
\begin{aligned}
\sigma_{\mathcal K_\iota}(\y)
={}&
\left(
 \underset{\lambda_m}{\operatorname{CT}}\,
 \frac{1}{\underline{1-\lambda_m^{s_m}\y^{\mathbf q_m}}}
\right)
\circ\cdots \qquad{}\circ
\left(
 \underset{\lambda_1}{\operatorname{CT}}\,
 \frac{1}{\underline{1-\lambda_1^{s_1}\y^{\mathbf q_1}}}
\right)
\frac{\Lambda^{-P\mathbf b}}
{\displaystyle\prod_{j\in J^c}
 (1-\Lambda^{PA_j}y_j)}.
\end{aligned}
\end{equation}
The displayed cofactor is regular on the common root set; hence every
selected-pole operator in \eqref{eq:successive-decd} is well defined.
\end{proposition}

It is easy to verify that
\[
 \Phi_{u_i,1}^{\lambda_i}H(\lambda_i)=H(u_i^{-1}).
\]
Since the operators  commute pairwise, we apply the above identity to all \(\Phi\)-operators with \(s_i = 1\). This allows us to assume without loss of generality that
\[
 S=\diag(s_1,\ldots,s_h),
 \qquad
 1<s_1\mid\cdots\mid s_h,
 \qquad
 \Lambda=(\lambda_1,\ldots,\lambda_h).
\]
Accordingly, the representations derived from Propositions~\ref{prop:simplicial-cone-ct} and \ref{prop:simpcone-smith} yield the following formulation, which serves as the starting point of our subsequent analysis:
\begin{equation}\label{eq:successive-ct-problem}
\left(
 \underset{\lambda_h}{\operatorname{CT}}\,
 \frac{1}{\underline{1-\lambda_h^{s_h}u_h}}
\right)
\circ\cdots\circ
\left(
 \underset{\lambda_1}{\operatorname{CT}}\,
 \frac{1}{\underline{1-\lambda_1^{s_1}u_1}}
\right)
\frac{c\,u_0\Lambda^{\mathbf c_0}}
{\displaystyle\prod_{\ell=1}^{r}
 \left(1-u_{h+\ell}\Lambda^{\mathbf c_\ell}\right)}.
\end{equation}

\begin{remark}
The output of \texttt{SimpCone[S]} consists of simplicial cones indexed by \(J\), an \(m\)-subset of \([n]\). Thus the number of terms is bounded by
\(\binom{n}{m}=\binom{n}{d}\), where \(d=n-m\) is the dimension. When \(d\) is fixed, this bound is polynomial in \(n\). Consequently, we may focus our attention on the decomposition of individual simplicial cones.
\end{remark}

\section{Local selected-pole calculus and the coprime \texttt{DecDenu} step}
\label{sec:local-calculus}

To jointly analyze the full operator structure, we first establish the connection between the root-average operator \(\Phi_{u_i,a_i}^{\lambda}\) and partial fraction decomposition for focused single-factor identities. We further introduce two fundamental operational procedures underlying our recursion framework: focused remainder reduction and pole exchange. 

Fix the variable \(\lambda\).  The coefficient field \(\mathbb L\) has
characteristic zero and contains the marker variables and all other binding
variables.  Thus every coefficient displayed below is independent of
\(\lambda\), and every \(u\)-symbol denotes a Laurent monomial coefficient.

The following proposition makes precise the relationship between the normalized root-average formula and partial-fraction decomposition.
\begin{proposition} 
\label{prop:contribution-average}
Let
\[
 F(\lambda)
 =
 \frac{u_0\lambda^{m_0}}
 {\prod_{i=1}^r(1-u_i\lambda^{a_i})},
 \qquad
 m_0\in\ZZ,\quad a_i\geq1,
\]
where \(u_0,u_i\in\mathbb L^*\), and suppose that the displayed binomials
are pairwise coprime.  Write the partial-fraction decomposition in
\(\lambda\) as
\[
 F(\lambda)
 =
 P(\lambda)+\frac{p(\lambda)}{\lambda^\ell}
 +\sum_{i=1}^r\frac{A_i(\lambda)}{1-u_i\lambda^{a_i}},
 \qquad
 \deg p<\ell,\quad \deg A_i<a_i,
\]
omitting \(p(\lambda)/\lambda^\ell\) when \(F\) has no pole at zero.
The contribution selected by \(1-u_i\lambda^{a_i}\) is \(A_i(0)\), where
\[
 A_i(\lambda)
 \equiv
 (1-u_i\lambda^{a_i})F(\lambda)
 \pmod{1-u_i\lambda^{a_i}}.
\]
Equivalently,
\begin{equation}\label{eq:contribution-average}
\begin{aligned}
 A_i(0)
 &=
 \Phi_{u_i,a_i}^{\lambda}
 \left(
  \frac{u_0\lambda^{m_0}}
  {\prod_{j\ne i}(1-u_j\lambda^{a_j})}
 \right) =
 \left(
  \underset{\lambda}{\operatorname{CT}}\,
  \frac{1}{\underline{1-u_i\lambda^{a_i}}}
 \right)
 \frac{u_0\lambda^{m_0}}
 {\displaystyle\prod_{j\ne i}(1-u_j\lambda^{a_j})}
 .
\end{aligned}
\end{equation}
\end{proposition}

\begin{proof}
Fix \(i\), and write
\[
 D(\lambda):=1-u_i\lambda^{a_i},
 \qquad
 H(\lambda):=D(\lambda)F(\lambda)
 =
 \frac{u_0\lambda^{m_0}}
 {\prod_{j\ne i}(1-u_j\lambda^{a_j})}.
\]
Pairwise coprimality implies that \(H\) is regular at every root of
\(D\).  Multiplying the partial-fraction decomposition by \(D\) and
reducing modulo \(D\) gives
\[
 A_i(\lambda)\equiv H(\lambda)\pmod{D(\lambda)}.
\]

Let \(\mathcal R_{a_i}(u_i)\) be the root set of \(D\).  Over
\(\overline{\mathbb L}\),
\[
 D(\lambda)
 =
 \prod_{\omega\in\mathcal R_{a_i}(u_i)}
 \left(1-\frac{\lambda}{\omega}\right).
\]
Since \(\deg A_i<a_i\), partial fractions over these simple roots give
\[
 \frac{A_i(\lambda)}{D(\lambda)}
 =
 \frac1{a_i}
 \sum_{\omega\in\mathcal R_{a_i}(u_i)}
 \frac{A_i(\omega)}{1-\lambda/\omega}.
\]
Indeed, the coefficient belonging to \(1-\lambda/\omega\) is
\[
 \left.
 \frac{(1-\lambda/\omega)A_i(\lambda)}{D(\lambda)}
 \right|_{\lambda=\omega}
 =
 \frac{A_i(\omega)}{-\omega D'(\omega)}
 =
 \frac{A_i(\omega)}{a_i},
\]
because \(u_i\omega^{a_i}=1\).

Setting \(\lambda=0\) and using \(A_i(\omega)=H(\omega)\), we obtain
\[
 A_i(0)
 =
 \frac1{a_i}
 \sum_{\omega\in\mathcal R_{a_i}(u_i)}H(\omega)
 =
 \Phi_{u_i,a_i}^{\lambda}H(\lambda),
\]
which is precisely \eqref{eq:contribution-average}.
\end{proof}

For \(a\geq1\) and \(q\in\ZZ\), let \(\angles{q}_a\) be the unique
integer satisfying
\begin{equation*}
 \angles{q}_a\equiv q\pmod a,
 \qquad
 \angles{q}_a\in(-a/2,a/2],
 \qquad
 [q]_a=\abs{\angles{q}_a}.
\end{equation*}
Thus \(0\leq[q]_a\leq a/2\).

\begin{proposition}[\cite{Xin2015}, Focused remainder reduction]
\label{prop:focused-reduction}
Fix the selected factor \(1-u\lambda^a\), and consider
\begin{equation*}
 F(\lambda)
 =
 \frac{u_0\lambda^{m_0}}
 {(1-u\lambda^a)\prod_{j=1}^r(1-u_j\lambda^{m_j})},
 \qquad
 m_0,m_j\in\ZZ,
\end{equation*}
where \(u_0,u,u_j\in\mathbb L^*\).  Assume that
\((1-u\lambda^a)F(\lambda)\) is regular at every root of
\(1-u\lambda^a\).  Put
\[
 b_j=[m_j]_a.
\]
The \(\lambda\)-free coefficient \(\overline u_j\) is determined directly
by the centered reduction
\begin{equation*}
 u_j\lambda^{m_j}
 \equiv
 \begin{cases}
  \overline u_j\lambda^{b_j},
   &\angles{m_j}_a\geq0,\\[2mm]
  \overline u_j^{-1}\lambda^{-b_j},
   &\angles{m_j}_a<0,
 \end{cases}
 \pmod{1-u\lambda^a}.
\end{equation*}
In the second case, use
\begin{equation*}
 \frac{1}{1-\overline u_j^{-1}\lambda^{-b_j}}
 =
 -\frac{\overline u_j\lambda^{b_j}}
        {1-\overline u_j\lambda^{b_j}},
\end{equation*}
and absorb the resulting sign and monomial into the numerator.  Reducing
the resulting numerator exponent modulo \(a\) gives
\begin{equation}\label{eq:focused-reduced-function}
 \overline F(\lambda)
 =
 \frac{\overline u_0\lambda^{b_0}}
 {(1-u\lambda^a)
  \prod_{j=1}^r(1-\overline u_j\lambda^{b_j})},
 \qquad
 1\leq b_0\leq a,\quad
 0\leq b_j\leq\frac a2.
\end{equation}
Here \(\overline u_0\) includes the scalar sign produced by the binomial
inversions.  If the ordinary numerator remainder is zero, we take
\(b_0=a\) rather than \(b_0=0\).  A factor with \(b_j=0\) is a
\(\lambda\)-independent factor and remains in the displayed product.  Then, the following congruence holds:
\begin{equation*}
 (1-u\lambda^a)\overline F(\lambda)
 \equiv
 (1-u\lambda^a)F(\lambda)
 \pmod{1-u\lambda^a}.
\end{equation*}
Thus \(F\) and \(\overline F\) have the same selected contribution:
$$
\Phi^{\lambda}_{u, a} (F(\lambda) (1-u \lambda^a)) =
\Phi^{\lambda}_{u, a} (\overline F(\lambda) (1-u \lambda^a)).
$$
In particular,
 when  $a=1$
we have
$$
\Phi^{\lambda}_{u, 1} (F(\lambda) (1-u \lambda)) =
F(\lambda) (1-u \lambda) \big|_{\lambda = u^{-1}}.
 $$

\end{proposition}

\begin{proposition}[\cite{Xin2015}, Pole exchange]\label{prop:pole-exchange}
Suppose that \(\overline F\) has the form
\eqref{eq:focused-reduced-function}, that some \(b_j \neq 0\), and that its
\(\lambda\)-dependent denominator binomials are pairwise coprime.  Then
\begin{equation}\label{eq:pole-exchange}
\begin{aligned}
&
 \underset{\lambda}{\operatorname{CT}}\,
 \frac{\overline u_0\lambda^{b_0}}
 {\underline{1-u\lambda^a}
  \prod_{\ell=1}^r
  (1-\overline u_\ell\lambda^{b_\ell})}\\
&\qquad
 =
 \sum_{\substack{1\leq j\leq r\\b_j>0}}
 \underset{\lambda}{\operatorname{CT}}\,
 \frac{-\overline u_0\lambda^{b_0}}
 {\underline{(1-\overline u_j\lambda^{b_j})}(1-u\lambda^a)
  \prod_{\ell\ne j}
  (1-\overline u_\ell\lambda^{b_\ell})}.
\end{aligned}
\end{equation}
\end{proposition}

Indeed, \(\overline F\) is proper and vanishes at zero, so the sum of all
its partial-fraction contributions is zero.   Thus every nonzero
remainder \(b_j\) produces one child with focused modulus
\(b_j\leq a/2\), whereas \(b_j=0\) produces no child.

The coprime multiplier identity recalled in the introduction applies
for \(k\geq1\) with \(\gcd(k,a)=1\)
\cite[Alg.\texttt{DecDenu}]{XinZhangZhang2024}.
One coprime \texttt{DecDenu} step consists of the following named
operations:
\[
\begin{aligned}
 \text{coprime multiplier transfer}
 &\longrightarrow \text{focused remainder reduction}
 \longrightarrow \text{pole exchange}.
\end{aligned}
\]
After the first operation, the focused remainder reduction uses
\(b_j=[km_j]_a\), and every nonzero \(b_j\) gives one pole child through
\eqref{eq:pole-exchange}.  The coprimality condition is essential:
if \(\gcd(k,a)>1\), exponentiation by \(k\) does not permute the \(a\)
roots.  The next section replaces the unavailable multiplier transfer by
an exact split.

All these operations act only on the focused variable \(\lambda\).
In the successive problem \eqref{eq:successive-ct-problem}, every other
root average remains attached to the same term.

\section{The non-coprime multiplier split}\label{sec:noncoprime}

The split is an identity for the one-variable pole operator and does not
depend on Smith normalization.  We therefore establish it before introducing
the form of a recursive term.

\begin{lemma}\label{lem:unit-lift}
Let $0<\abs{k}<s$ and put
\begin{equation*}
 g=\gcd(k,s),\qquad b=s/g,\qquad
 \kappa_0\equiv k/g\pmod b.
\end{equation*}
Using gcd and the Chinese remainder theorem, one can compute an integer
$\kappa$ such that
\begin{equation*}
 1\leq\kappa<s,\qquad
 \kappa\equiv\kappa_0\pmod b,\qquad
 \gcd(\kappa,s)=1.
\end{equation*}
No prime factorization of $s$ is required.
\end{lemma}

\begin{proof}
Set $h_0=\gcd(g,b)$ and iterate
\begin{equation*}
 h_{i+1}=\gcd(g,h_i^2)
\end{equation*}
until the value stabilizes; call the final value $g_\infty$.  For a prime dividing
both $g$ and $b$, its exponent in $h_i$ doubles until it reaches its full
exponent in $g$.  Thus $g_\infty$ is the largest divisor of $g$ supported
on the primes dividing $b$, and $g'=g/g_\infty$ is coprime to $b$.  The iteration takes
$O(\log\log s)$ gcd computations.

Because $\gcd(\kappa_0,b)=1$, solve
\[
 \kappa\equiv\kappa_0\pmod b,
 \qquad
 \kappa\equiv1\pmod{g'}
\]
by the Chinese remainder theorem.  Its least positive representative is
smaller than $bg'\leq bg=s$.  Primes of $g$ that also divide $b$ do not
divide $\kappa$, and primes in $g'$ do not divide $\kappa$ either.  Hence
$\gcd(\kappa,s)=1$.
\end{proof}

\begin{lemma}[Exact non-coprime split]\label{lem:split}
Use the notation of Lemma~\ref{lem:unit-lift}. Since $\gcd(\kappa,s)=1$, we choose
$1\leq e<s$ and $t\in\ZZ$ such that
\begin{equation*}
 e\kappa=1+ts.
\end{equation*}
%The $s$ roots of $u\lambda^s=1$ are in bijection with the common roots of
%\begin{equation}\label{eq:split-equations}
% \nv_2^b=\nv_1,\qquad u^e\nv_1^g=1,
%\end{equation}
%through
%\begin{equation}\label{eq:split-substitution}
% \lambda=u^t\nv_2^\kappa.
%\end{equation}
Consequently, whenever \(H(\lambda)\in\mathbb L(\lambda)\) is defined at
these roots,
\begin{equation*}
 \Phi_{u,s}^{\lambda}H(\lambda)
 =\Phi_{u^e,g}^{\nv_1}\Phi_{\nv_1^{-1},b}^{\nv_2}
   H(u^t\nv_2^\kappa).
\end{equation*}
All exponents in this identity are \emph{integers}.
\end{lemma}

\begin{proof}
The denominator of \(H\) is relatively prime to \(1-u\lambda^s\), and is
therefore invertible in
\(\mathbb L[\lambda,\lambda^{-1}]/(1-u\lambda^s)\).  Hence \(H\) has a
Laurent-polynomial representative on the root set.  By linearity, it
suffices to consider the case
\[
H(\lambda)=\lambda^q.
\]
Using the standard root-of-unity filter, we obtain
\[
\mathrm{LHS}
 =\frac{1}{s}\sum_{\omega^s=u^{-1}}\omega^q
 =
 \begin{cases}
  u^{-q/s}, & s\mid q,\\
  0,        & s\nmid q.
 \end{cases}
\]

We now compute the right-hand side:
\[
\mathrm{RHS}
=
\Phi_{u^e,g}^{\nv_1}
\Phi_{\nv_1^{-1},b}^{\nv_2}
\left(u^{tq}\nv_2^{\kappa q}\right).
\]
Applying the inner operator first gives
\[
\begin{aligned}
\Phi_{\nv_1^{-1},b}^{\nv_2}
\left(u^{tq}\nv_2^{\kappa q}\right)
&=
\frac{u^{tq}}{b}
\sum_{\omega_2^b=\nv_1}\omega_2^{\kappa q} =
\begin{cases}
u^{tq}\nv_1^{\kappa q/b}, & b\mid \kappa q,\\
0,                         & b\nmid \kappa q.
\end{cases}
\end{aligned}
\]

Suppose first that \(b\nmid\kappa q\). Since
\(\gcd(\kappa,s)=1\) and \(b\mid s\), we have
\(\gcd(\kappa,b)=1\). Thus \(b\nmid q\), and hence \(s\nmid q\).
Consequently, both sides are zero.

Now suppose that \(b\mid\kappa q\). Applying the outer operator yields
\[
\begin{aligned}
\mathrm{RHS}
&=
\frac{u^{tq}}{g}
\sum_{\omega_1^g=u^{-e}}
\omega_1^{\kappa q/b} =
\begin{cases}
u^{tq}(u^{-e})^{\kappa q/(bg)},
  & g\mid \dfrac{\kappa q}{b},\\[6pt]
0,
  & g\nmid \dfrac{\kappa q}{b}.
\end{cases}
\end{aligned}
\]

If \(g\nmid \kappa q/b\), then \(s=bg\nmid\kappa q\). Since
\(\gcd(\kappa,s)=1\), it follows that \(s\nmid q\), so once again both
sides vanish.

Finally, suppose that
\[
g\mid\frac{\kappa q}{b}.
\]
Then \(s=bg\mid\kappa q\), and therefore \(s\mid q\). Using
\(e\kappa=1+ts\) and \(s=bg\), we have
\[
\frac{e\kappa q}{bg}
=
\frac{e\kappa q}{s}
=
\frac{(1+ts)q}{s}
=
\frac{q}{s}+tq.
\]
Hence
\[
\begin{aligned}
\mathrm{RHS}
&=
u^{tq}(u^{-e})^{\kappa q/(bg)}=
u^{tq-e\kappa q/s}=
u^{tq-(q/s+tq)}
=
u^{-q/s},
\end{aligned}
\]
which agrees with the left-hand side. Therefore the desired identity
holds.
\end{proof}

\section{Joint root averages and recursive normalization}
\label{sec:recursive-terms}

We define the standard normal form of a non-terminal term as
\begin{equation}\label{eq:normal-term}
 \Phi_{u_1,s_1}^{\lambda_1}\cdots
 \Phi_{u_h,s_h}^{\lambda_h}
 \left(
  \frac{u_0\Lambda^{\mathbf c_0}}
       {\prod_{\ell=1}^{r}
        \left(1-u_{h+\ell}\Lambda^{\mathbf c_\ell}\right)}
 \right),
\end{equation}
where
\begin{equation*}
 S=\diag(s_1,\ldots,s_h),
 \qquad
 1<s_1\mid s_2\mid\cdots\mid s_h,
\end{equation*}
and
\[
 U=(u_1,\ldots,u_h),
 \qquad
 \Lambda=(\lambda_1,\ldots,\lambda_h).
\]
The exponent vectors
\(\mathbf c_0,\mathbf c_1,\ldots,\mathbf c_r\in\mathbb Z^h\)
are normalized coordinatewise. We use positive residues for the
numerator exponent and centered residues for the denominator exponents:
\begin{equation}\label{eq:normalized-centered-range}
 1\leq(\mathbf c_0)_i\leq s_i,
 \qquad
 -\frac{s_i}{2}
 <
 (\mathbf c_\ell)_i
 \leq
 \frac{s_i}{2}
 \quad
 (1\leq i\leq h,\ 1\leq\ell\leq r).
\end{equation}

Let
\[
 C:=(\mathbf c_1,\ldots,\mathbf c_r),
 \qquad
 D:=\abs{\det S}=s_1\cdots s_h.
\]
Whenever these normalization conditions are not satisfied, use
\(\lambda_i^{s_i}=u_i^{-1}\) to reduce the exponents to the displayed
ranges and absorb the resulting marker factors into the coefficients.
Here and below, applications of Proposition~\ref{prop:focused-reduction}
for this normalization omit binomial inversions.

If \(h=0\), the resulting Laurent rational function in the marker
variables alone is an output term of the recursion.

\subsection{Joint realization and normalization of a pole-exchange child}

Consider the nonterminal normal-form term \eqref{eq:normal-term}.  Focus on the last operator
\(\Phi_{u_h,s_h}^{\lambda_h}\), and abbreviate
\[
\begin{gathered}
 \lambda:=\lambda_h,
 \qquad
 u:=u_h,
 \qquad
 s:=s_h,
 \qquad
 \Lambda_{h-1}:=(\lambda_1,\ldots,\lambda_{h-1}),\\
 \mathbf c_\ell
 =
 \begin{pmatrix}
  \mathbf c'_\ell\\
  \gamma_\ell
 \end{pmatrix},
 \qquad
 \gamma_\ell:=(\mathbf c_\ell)_h
 \quad(0\leq\ell\leq r),
 \qquad
 \boldsymbol\gamma:=(\gamma_1,\ldots,\gamma_r).
\end{gathered}
\]

Choose an integer \(k\) satisfying \(0<\abs{k}<s\).  Put
\(g=\gcd(k,s)\) and \(b=s/g\), and let \(\kappa,e,t\) be the integers
supplied by Lemma~\ref{lem:split}.  The exact split gives
\begin{equation}\label{eq:focused-split-recalled}
 \Phi_{u,s}^{\lambda}G(\Lambda)
 =
 \Phi_{u^e,g}^{\nv_h}
 \Phi_{\nv_h^{-1},b}^{\nv_{h+1}}
 G(\Lambda_{h-1},u^t\nv_{h+1}^{\kappa}),
\end{equation}
where
\[
 G\bigl(\Lambda_{h-1},u^t\nv_{h+1}^{\kappa}\bigr)
 =
 \frac{u_0\Lambda_{h-1}^{\mathbf c'_0}
       (u^t\nv_{h+1}^{\kappa})^{\gamma_0}}
 {\displaystyle\prod_{\ell=1}^{r}
  \left(1-u_{h+\ell}\Lambda_{h-1}^{\mathbf c'_\ell}
       (u^t\nv_{h+1}^{\kappa})^{\gamma_\ell}\right)}.
\]

Apply Proposition~\ref{prop:focused-reduction} to the inner operator in
\eqref{eq:focused-split-recalled}, and use its output notation
\begin{equation*}
\begin{aligned}
&\Phi_{\nv_h^{-1},b}^{\nv_{h+1}}
 G\bigl(\Lambda_{h-1},u^t\nv_{h+1}^{\kappa}\bigr)={}
 \Phi_{\nv_h^{-1},b}^{\nv_{h+1}}
 \bar G(\Lambda_{h-1},\nv_h,\nv_{h+1}),
\end{aligned}
\end{equation*}
where
\begin{equation*}
 \bar G(\Lambda_{h-1},\nv_h,\nv_{h+1})
 =
 \frac{w_0\nv_{h+1}^{p_0}}
 {\displaystyle\prod_{\ell=1}^{r}
  \left(1-w_{h+\ell}\nv_{h+1}^{p_\ell}\right)},
 \qquad
 p_\ell=[\kappa\gamma_\ell]_b.
\end{equation*}
Here \(1\leq p_0\leq b\), \(0\leq p_\ell\leq b/2\), and
\(w_0,w_{h+1},\ldots,w_{h+r}\) are Laurent monomials in the marker
variables $\y$ and the outer binding variables \(\Lambda_{h-1},\nv_h\),
independent of \(\nv_{h+1}\).  These symbols denote exactly the data
produced by Proposition~\ref{prop:focused-reduction}.
For later reference, we record
\begin{equation}\label{eq:child-focused-order}
 p_\ell=[\kappa\gamma_\ell]_b.
\end{equation}

For a recursive state, at least one \(p_j\) is positive.  Indeed,
\(b\geq2\) and \(\gcd(\kappa,b)=1\); if every \(p_j\) vanished, then
\(b\) would divide \(s,\gamma_1,\ldots,\gamma_r\), contrary to
\eqref{eq:focused-primitivity}.

The regularity and coprimality assumptions of
Proposition~\ref{prop:pole-exchange} will be verified for recursive terms
in Lemma~\ref{lem:regularity-coprimality}.  That proposition gives
\begin{equation*}
 \Phi_{\nv_h^{-1},b}^{\nv_{h+1}}
 \bar G(\Lambda_{h-1},\nv_h,\nv_{h+1})
 =
 \sum_{p_j>0}
 \Phi_{w_{h+j},p_j}^{\nv_{h+1}}
 H_j(\Lambda_{h-1},\nv_h,\nv_{h+1}),
\end{equation*}
where
\begin{equation*}
 H_j(\Lambda_{h-1},\nv_h,\nv_{h+1})
 :=
 -\frac{\left(1-w_{h+j}\nv_{h+1}^{p_j}\right)
          \bar G(\Lambda_{h-1},\nv_h,\nv_{h+1})}
        {1-\nv_h^{-1}\nv_{h+1}^{b}}.
\end{equation*}

Relabel the ordered raw binding tuple as
\[
 \V=(\nv_1,\ldots,\nv_{h+1})
 :=
 (\lambda_1,\ldots,\lambda_{h-1},\nv_h,\nv_{h+1}),
 \qquad
 \V_{h-1}:=(\nv_1,\ldots,\nv_{h-1}).
\]
For each \(j\) with \(p_j>0\), write the selected reduced coefficient
uniquely as
\begin{equation*}
 w_{h+j}
 =
 \widehat u_{h+j}
 \V_{h-1}^{\mathbf{\eta}_j}
 \nv_h^{\tau_j},
 \qquad
 \widehat u_{h+j}\ \text{marker-only},\quad
 \mathbf{\eta}_j\in\ZZ^{h-1},\quad
 \tau_j\in\ZZ.
\end{equation*}
Combining the outer operators, the corresponding child is
\begin{equation}\label{eq:pole-exchange-child-iterated}
 \mathcal C_j
 :={}
 \Phi_{u_1,s_1}^{\nv_1}\cdots
 \Phi_{u_{h-1},s_{h-1}}^{\nv_{h-1}}
 \Phi_{u^e,g}^{\nv_h}
 \Phi_{
  \widehat u_{h+j}
  \V_{h-1}^{\mathbf{\eta}_j}\nv_h^{\tau_j},
  p_j
 }^{\nv_{h+1}}
 H_j(\V).
\end{equation}
Here \(0<p_j\leq b/2\).  The operators in
\eqref{eq:pole-exchange-child-iterated} act from right to left.

In conclusion, we obtain
\[
\Phi_{u_1,s_1}^{\lambda_1}\cdots
 \Phi_{u_h,s_h}^{\lambda_h}G(\Lambda)  =  \sum_{p_j >0}  \mathcal C_j.
\]

  A factor
with \(p_j=0\) remains in the common cofactor and creates no child;
different nonzero summands are normalized separately.

\subsection{Joint root averages and raw recursive terms}

Let \(q\geq1\), let
\[
 \V=(\nv_1,\ldots,\nv_q),
 \qquad
 U=(u_1,\ldots,u_q),
\]
and suppose that every \(u_i\) is a nonzero Laurent monomial in the marker
variables.  Let
\[
 B=(\boldsymbol\beta_1,\ldots,\boldsymbol\beta_q)
 \in\ZZ^{q\times q}
\]
be nonsingular, with exponent columns
\(\boldsymbol\beta_i\in\ZZ^q\).  The associated carried equations are
\begin{equation}\label{eq:general-pole-system}
 u_i\V^{\boldsymbol\beta_i}=1
 \qquad(1\leq i\leq q),
\end{equation}
and their common-root set is
\begin{equation*}
 \mathcal R_B(U)
 :=
 \left\{
  \boldsymbol\omega\in(\overline\K^{\,*})^q:
  u_i\boldsymbol\omega^{\boldsymbol\beta_i}=1
  \text{ for }1\leq i\leq q
 \right\}.
\end{equation*}

\begin{definition}[Joint root average]\label{def:joint-phi}
If a rational function \(H(\V)\) is defined at every point of
\(\mathcal R_B(U)\), set
\begin{equation}\label{eq:joint-root-average}
 \Psi_{U,B}^{\V}H(\V)
 :=
 \frac{1}{\abs{\det B}}
 \sum_{\boldsymbol\omega\in\mathcal R_B(U)}
 H(\boldsymbol\omega).
\end{equation}
\end{definition}

When \(B=S=\diag(s_1,\ldots,s_q)\), the equations are independent and
\[
 \mathcal R_S(U)
 =
 \mathcal R_{s_1}(u_1)\times\cdots\times
 \mathcal R_{s_q}(u_q).
\]
Consequently, finite summation and Definition~\ref{def:joint-phi} give
\begin{equation*}
 \Psi_{U,S}^{\Lambda}H(\Lambda)
 =
 \Phi_{u_1,s_1}^{\lambda_1}\cdots
 \Phi_{u_q,s_q}^{\lambda_q}H(\Lambda),
 \qquad
 \Lambda=(\lambda_1,\ldots,\lambda_q).
\end{equation*}
 Its value is independent of the order. For nondiagonal systems with coupled variables, the joint average cannot be decomposed into commuting one-variable averages. When expressed as nested one-variable averages, the rightmost operator is applied first.

Suppose
$\Psi_{U,B}^{\V} F(\V)$
 is regular, i.e.,
that $F(\V)$ is well defined on \(\mathcal R_B(U)\).
We use the Smith normal form theorem: for every nonsingular integer matrix \(B\), one can compute
\(P,Q\in\operatorname{GL}_q(\ZZ)\) such that \(PBQ\) is diagonal with
positive divisibility-ordered entries; see \cite{KannanBachem1979}.  The
following result records the induced reindexing of the finite root
average.

\begin{theorem}[Smith normalization of a joint root average]
\label{thm:smith-normalization}
Let \(B\in\ZZ^{q\times q}\) be nonsingular, and choose
\(P,Q\in\operatorname{GL}_q(\ZZ)\) such that
\begin{equation}\label{eq:smith-reduction}
 PBQ=S:=
 \diag(s_1,\ldots,s_q),
 \qquad
 1\leq s_1\mid s_2\mid\cdots\mid s_q.
\end{equation}
For \(1\leq j\leq q\), define
\begin{equation*}
 \widetilde u_j := U^{Q\mathbf e_j} = \prod_{i=1}^q u_i^{Q_{ij}},
 \qquad
 \widetilde U := (\widetilde u_1,\ldots,\widetilde u_q).
\end{equation*}
Let \(\Lambda=(\lambda_1,\ldots,\lambda_q)\), and define
\begin{equation*}
 \V_P(\Lambda)
 :=
 \left(
  \Lambda^{P\mathbf e_1},\ldots,
  \Lambda^{P\mathbf e_q}
 \right).
\end{equation*}
Then the system \eqref{eq:general-pole-system} is carried bijectively to
\begin{equation*}
 \widetilde u_j\lambda_j^{s_j}=1
 \qquad(1\leq j\leq q).
\end{equation*}
In particular,
\[
 \abs{\mathcal R_B(U)}
 =
 \abs{\det B}
 =
 s_1\cdots s_q.
\]
For every rational function \(H\) defined on \(\mathcal R_B(U)\), one
has
\begin{equation}\label{eq:smith-average}
 \Psi_{U,B}^{\V}H(\V)
 =
 \Psi_{\widetilde U,S}^{\Lambda}
 H\bigl(\V_P(\Lambda)\bigr).
\end{equation}
\end{theorem}

\begin{proof}
Right multiplication by \(Q\) replaces the original equations by the
unimodular multiplicative combinations
\begin{equation*}
 \prod_{i=1}^q
 \left(u_i\V^{\boldsymbol\beta_i}\right)^{Q_{ij}}
 =1
 \qquad(1\leq j\leq q).
\end{equation*}
Because \(Q^{-1}\) also has integer entries, the original equations are
integral multiplicative combinations of these new equations.  The two
systems therefore have the same common-root set.

The substitution \(\V=\V_P(\Lambda)\) is an automorphism of
\((\overline\K^{\,*})^q\).  Its inverse is the monomial map
\[
 \lambda_k=\V^{P^{-1}\mathbf e_k}
 \qquad(1\leq k\leq q).
\]
Under this substitution, the \(j\)th combined equation becomes
\[
 \widetilde u_j\Lambda^{P(BQ)\mathbf e_j}
 =
 \widetilde u_j\lambda_j^{s_j}
 =1,
 \qquad
 P(BQ)\mathbf e_j=S\mathbf e_j=s_j\mathbf e_j.
\]
Thus the monomial automorphism gives the asserted bijection with the
diagonal root set.

The \(j\)th diagonal equation has exactly \(s_j\) distinct nonzero roots,
so the diagonal system has \(s_1\cdots s_q\) common roots.  Taking
determinants in \eqref{eq:smith-reduction} gives
\[
 s_1\cdots s_q=\abs{\det B},
\]
because \(\abs{\det P}=\abs{\det Q}=1\).  Finally, reindexing the finite
sum in \eqref{eq:joint-root-average} by the root-set bijection proves
\eqref{eq:smith-average}.
\end{proof}

For the child in \eqref{eq:pole-exchange-child-iterated},
\[
 \mathcal C_j
 :={}
 \Phi_{u_1,s_1}^{\nv_1}\cdots
 \Phi_{u_{h-1},s_{h-1}}^{\nv_{h-1}}
 \Phi_{u^e,g}^{\nv_h}
 \Phi_{
  \widehat u_{h+j}
  \V_{h-1}^{\mathbf{\eta}_j}\nv_h^{\tau_j},
  p_j
 }^{\nv_{h+1}}
 H_j(\V).
\]
With columns ordered as the carried equations and rows ordered as \(\V\),
the raw exponent matrix is
\begin{equation*}
\begin{aligned}
 B_j^{\rm raw}
 &:=
 \begin{pmatrix}
  S_{h-1}&0&\mathbf{\eta}_j\\
  0&g&\tau_j\\
  0&0&p_j
 \end{pmatrix}, \quad
 S_{h-1}:=\diag(s_1,\ldots,s_{h-1}).
\end{aligned}
\end{equation*}
The corresponding marker-only carried coefficient tuple is
\begin{equation*}
 U_j^{\rm raw}
 :=
 (u_1,\ldots,u_{h-1},u^e,\widehat u_{h+j}).
\end{equation*}
In particular,
\begin{equation}\label{eq:raw-child-determinant}
 \abs{\det B_j^{\rm raw}}
 =
 \left(\prod_{i=1}^{h-1}s_i\right)g p_j
 =\frac{D}{s_h}g[\kappa\gamma_j]_b
 =\frac{D}{s_h}[k\gamma_j]_{s_h}.
\end{equation}
Here \(g p_j=[k\gamma_j]_s\) follows from \(s=gb\) and
\(\kappa\equiv k/g\pmod b\).

\begin{proposition}
\label{prop:child-joint-realization}
Assume that \(H_j\) is defined on
\(\mathcal R_{B_j^{\rm raw}}(U_j^{\rm raw})\).  Then the iterated average
\(\mathcal C_j\) coincides with the joint root average
\[
 \mathcal C_j
 =
 \Psi_{U_j^{\rm raw},B_j^{\rm raw}}^{\V}
 H_j(\V).
\]
Moreover, choose \(P_j,Q_j\in\operatorname{GL}_{h+1}(\ZZ)\) so that
\[
 S_j:=P_jB_j^{\rm raw}Q_j
\]
is in Smith normal form.  Define the transformed carried coefficients
\(U_j=(u_{j,1},\ldots,u_{j,h+1})\) by
\begin{equation*}
 u_{j,k}
 :=
 \bigl(U_j^{\rm raw}\bigr)^{Q_j\mathbf e_k}
 \qquad(1\leq k\leq h+1).
\end{equation*}
For a fresh tuple \(\Lambda=(\lambda_1,\ldots,\lambda_{h+1})\),
Theorem~\ref{thm:smith-normalization} gives
\begin{equation*}
 \mathcal C_j
 =
 \Psi_{U_j^{\rm raw},B_j^{\rm raw}}^{\V}H_j(\V)
 =
 \Psi_{U_j,S_j}^{\Lambda}
 H_j\bigl(\V_{P_j}(\Lambda)\bigr).
\end{equation*}

\end{proposition}

\begin{proof}
Let
\[
 \mathcal R_{\rm out}
 :=
 \left(\prod_{i=1}^{h-1}\mathcal R_{s_i}(u_i)\right)
 \times\mathcal R_g(u^e).
\]
For
\(\boldsymbol\omega_{\rm out}=(\omega_1,\ldots,\omega_h)
\in\mathcal R_{\rm out}\), let
\[
 \mathcal F_j(\boldsymbol\omega_{\rm out})
 :=
 \left\{
  \omega_{h+1}\in\overline\K^{\,*}:
  \widehat u_{h+j}
  \prod_{i=1}^{h-1}
  \omega_i^{(\mathbf{\eta}_j)_i}
  \omega_h^{\tau_j}\omega_{h+1}^{p_j}=1
 \right\}.
\]
The defining equation has exactly \(p_j\) distinct nonzero roots.  Since
\(B_j^{\rm raw}\) is upper triangular, its common-root set is the disjoint
union of these fibres over \(\boldsymbol\omega_{\rm out}\in\mathcal R_{\rm out}\).

The hypothesis ensures that every summand below is defined.  Expanding the
iterated averages from right to left and using
\eqref{eq:raw-child-determinant}, we obtain
\[
\begin{aligned}
 \mathcal C_j
 &=
 \frac{1}{\left(\prod_{i=1}^{h-1}s_i\right)g p_j}
 \sum_{\boldsymbol\omega_{\rm out}\in\mathcal R_{\rm out}}
 \sum_{\omega_{h+1}\in
       \mathcal F_j(\boldsymbol\omega_{\rm out})}
 H_j(\boldsymbol\omega_{\rm out},\omega_{h+1})\\
 &=
 \frac{1}{\abs{\det B_j^{\rm raw}}}
 \sum_{\boldsymbol\omega\in
       \mathcal R_{B_j^{\rm raw}}(U_j^{\rm raw})}
 H_j(\boldsymbol\omega)\\
 &=
 \Psi_{U_j^{\rm raw},B_j^{\rm raw}}^{\V}H_j(\V).
\end{aligned}
\]
The Smith-normalized identity follows from
Theorem~\ref{thm:smith-normalization}.
\end{proof}

\paragraph{\textbf{Return to standard form.}}
The form
\[
\Psi_{U_j,S_j}^{\Lambda} H_j\bigl(\V_{P_j}(\Lambda)\bigr)
\]
from Proposition~\ref{prop:child-joint-realization} may not be a standard-form term and may contain unit entries. We reduce it using Proposition~\ref{prop:focused-reduction}. If no binding variable remains, the result is a marker-only output term; otherwise, it is a nonterminal standard-form term~\eqref{eq:normal-term}.

\begin{remark}

(1)
In this section, the recursive step consists
of the following operations:
\[
\begin{aligned}
 &\text{short multiplier}
 \longrightarrow \text{exact split}\longrightarrow \text{focused remainder reduction}
 \longrightarrow \text{pole exchange}\\
 &\longrightarrow \text{Smith normalization}
 \longrightarrow \text{remainder reduction}.
\end{aligned}
\]

(2) Computing the Smith normal form for the specially structured matrix $B_j^{\rm raw}$ is much faster than applying a general-purpose Smith normal-form algorithm. This, however, amounts to a detailed optimization; the generic complexity bound is already sufficient to establish our complexity theorems.

(3) The nonsingular raw exponent matrices arising in the present recursion are
upper triangular.  In the corresponding nested root averages, each inner
average is taken fibrewise over the outer root set and therefore leaves that
root set unchanged, although it cannot in general be interchanged with the
outer averages.  When \(B\) is diagonal, the root equations are independent,
so the corresponding one-variable \(\Phi\)-operators commute and may be
applied in any order.

\end{remark}

\section{Algorithm}
\label{sec:jointct-algorithm}

The preceding sections provide all local identities needed for the
recursion.  We now assemble them into a complete algorithm for a general
full-row-rank system without introducing a second set of recursive-state
notation.

\subsection{Primitive row-lattice reduction}

\begin{proposition}[Primitive row-lattice reduction]
\label{prop:primitive-reduction}
Let \(A\in\ZZ^{m\times n}\) have rank \(m\), and let
\(\mathbf b\in\ZZ^m\).  In deterministic polynomial time, one can either
certify that \(A\mathbf x=\mathbf b\) has no integral solution or replace
\((A,\mathbf b)\) by an integral pair
\((\widetilde A,\widetilde{\mathbf b})\) such that
\[
 \mathcal P(A,\mathbf b)
 =\mathcal P(\widetilde A,\widetilde{\mathbf b}),
 \qquad
 \sigma_{ \mathcal P( A,\mathbf b)}(\y)
 =\sigma_{\mathcal P(\widetilde A,\widetilde{\mathbf b} )}(\y),
\]
and
\begin{equation*}
 \widetilde A\ZZ^n=\ZZ^m
 \quad\Longleftrightarrow\quad
 \delta_m(\widetilde A)=1,
\end{equation*}
where \(\delta_m\) is the maximal determinantal divisor.
\end{proposition}

\begin{proof}
This is a standard consequence of the Smith normal form; for a detailed
combinatorial treatment we refer to
\cite{XinZhangZhang2024}.  We include a self-contained proof here for
completeness and to keep the algorithm independent of external
combinatorial machinery.

Choose \(P\in\operatorname{GL}_m(\ZZ)\) and
\(Q\in\operatorname{GL}_n(\ZZ)\) such that
\[
 PAQ=(S\ \ 0),
 \qquad
 S=\diag(s_1,\ldots,s_m),
 \qquad
 1\leq s_1\mid\cdots\mid s_m.
\]
If some \(s_i\) does not divide \((P\mathbf b)_i\), then the equation has
no integral solution.  Otherwise set
\[
 \widetilde A:=S^{-1}PA,
 \qquad
 \widetilde{\mathbf b}:=S^{-1}P\mathbf b.
\]
Both are integral,
\(\widetilde A=(I_m\ \ 0)Q^{-1}\), and
\[
 A\mathbf x=\mathbf b
 \quad\Longleftrightarrow\quad
 \widetilde A\mathbf x=\widetilde{\mathbf b}.
\]

This proves the two identities and
\(\widetilde A\ZZ^n=\ZZ^m\); the equivalence with the determinantal
divisor is the standard Smith-form characterization of the column-lattice
index.  A
transformation-producing Smith form and all displayed arithmetic have
polynomial bit complexity and polynomial output size
\cite{KannanBachem1979}.
\end{proof}

Whenever the second case of Proposition~\ref{prop:primitive-reduction}
occurs, we replace the input by that pair and continue to denote it by
\(A,\mathbf b\).  Thus, throughout the algorithm,
\begin{equation*}
 A\ZZ^n=\ZZ^m,
 \qquad
 \delta_m(A)=1.
\end{equation*}

\begin{remark}[Initial generation condition]
\label{rem:initial-generation}
Write the two constant-term inputs of Section~\ref{sec:general-reduction}
in the raw-state notation of Section~\ref{sec:recursive-terms}.  For a
\texttt{SimpCone[S]} term, before centered reduction,
\[
 B=S=PA_JQ,
 \qquad
 C=(PA_j)_{j\in J^c}.
\]
Consequently,
\[
 [\,B\ C\,]\ZZ^n
 =P[\,A_JQ\ A_{J^c}\,]\ZZ^n
 =PA\ZZ^n
 =\ZZ^m,
\]
and hence
\[
 \delta_m([\,B\ C\,])=1.
\]

For the direct simplicial-cone expression
\eqref{eq:simplicial-cone-ct}, the carried block in the notation of
Section~\ref{sec:recursive-terms} is the displayed Smith matrix \(B=S\),
whereas the remaining exponent block is \(C=-P\).  Since \(P\) is
unimodular,
\[
 \delta_d([\,B\ C\,])=1.
\]

These two facts serve as the base cases for the induction argument that
establishes preservation of the generation condition in
Theorem~\ref{thm:generation-preserved}.
\end{remark}

\subsection{Short multipliers and the queue algorithm}

We first record the conditional arithmetic procedure used at a normalized
state.  Its hypothesis will be verified for every state in
Lemma~\ref{lem:h-le-r}.

\begin{lemma}\label{lem:minkowski}
Fix \(r\geq1\), \(s>1\), and
\(\gamma_1,\ldots,\gamma_r\in\ZZ\), and suppose that
\begin{equation}\label{eq:multiplier-primitivity}
 \gcd(s,\gamma_1,\ldots,\gamma_r)=1.
\end{equation}
There exist \(k\in\ZZ\) and a nonzero vector
\(\boldsymbol\rho=(\rho_1,\ldots,\rho_r)\in\ZZ^r\) such that
\begin{equation}\label{eq:minkowski-bound}
 1\leq k\leq\left\lfloor\frac{s}{2}\right\rfloor,
 \qquad
 \rho_j\equiv k\gamma_j\pmod s,
 \qquad
 \norm{\boldsymbol\rho}_\infty\leq s^{(r-1)/r}.
\end{equation}
Consequently,
\begin{equation}\label{eq:k-range}
 [k\gamma_j]_s\leq s^{(r-1)/r}
 \qquad(1\leq j\leq r).
\end{equation}
\end{lemma}

\begin{proof}
The existence of such a multiplier was proved in Proposition~4 of \cite{XinZhangZhang2024} using the pigeonhole principle, which yields a slightly weaker bound.
We reproduce the proof here with a slightly different lattice
viewpoint, which also makes the connection to Minkowski's theorem
transparent and prepares for the complexity discussion in
Remark~\ref{rem:LLLm}.

Let
\[
 \cL_{\boldsymbol\gamma,s}
 =\ZZ\boldsymbol\gamma+s\ZZ^r\subseteq\ZZ^r,
 \qquad
 \boldsymbol\gamma=(\gamma_1,\ldots,\gamma_r).
\]
Equivalently, this is the rank-\(r\) lattice generated by the rows of
\begin{equation*}
 \begin{pmatrix}
  \gamma_1&\gamma_2&\cdots&\gamma_r\\
  s&0&\cdots&0\\
  0&s&\cdots&0\\
  \vdots&\vdots&\ddots&\vdots\\
  0&0&\cdots&s
 \end{pmatrix}.
\end{equation*}
Put \(g_0=\gcd(s,\gamma_1,\ldots,\gamma_r)\).  The class of
\(\boldsymbol\gamma\) in \((\ZZ/s\ZZ)^r\) has order \(s/g_0\).
Reduction modulo \(s\) therefore gives
\[
 [\ZZ^r:\cL_{\boldsymbol\gamma,s}]
 =\frac{s^r}{s/g_0}=g_0s^{r-1}.
\]
Under \eqref{eq:multiplier-primitivity}, it follows that
\[
 \det\cL_{\boldsymbol\gamma,s}=s^{r-1}.
\]

Set \(R_s=s^{(r-1)/r}\) and
\(\mathcal B_s=[-R_s,R_s]^r\).  Then
\[
 \operatorname{vol}(\mathcal B_s)
 =(2R_s)^r=2^rs^{r-1}=2^r\det\cL_{\boldsymbol\gamma,s}.
\]
So, Minkowski's convex-body theorem
\cite[Chapter~III, Theorem~I]{Cassels1997}
 states that
there is at
least one nonzero lattice point of the closed box
\(\mathcal B_s\),  i.e.,
\[
 0\ne\boldsymbol\rho\in\cL_{\boldsymbol\gamma,s},
 \qquad
 \norm{\boldsymbol\rho}_\infty\leq R_s.
\]

Write \(\boldsymbol\rho=k_0\boldsymbol\gamma+s\mathbf m\).  Since
\(R_s<s\), the coefficient \(k_0\) is not divisible by \(s\): otherwise
\(0\ne\boldsymbol\rho\in s\ZZ^r\) would imply
\(\|\boldsymbol\rho\|_\infty\geq s\).  Reduce \(k_0\)
to \(k\in\{1,\ldots,s-1\}\); if \(k>s/2\), replace
\((k,\boldsymbol\rho)\) by \((s-k,-\boldsymbol\rho)\).  This proves
\eqref{eq:minkowski-bound}, and \eqref{eq:k-range} follows from the
definition of \([k\gamma_j]_s\).
\end{proof}

\begin{remark} \label{rem:LLLm}
  For fixed \(r\), such a multiplier can be found in time polynomial in the
binary input length by Lenstra's fixed-dimensional integer-programming algorithm  \cite{Lenstra1983}.

In practice, we recommend first using the LLL algorithm
\cite{LenstraLenstraLovasz1982} to compute an LLL-reduced basis of
$\cL_{\boldsymbol\gamma,s}$. A candidate multiplier is used only if
it satisfies \eqref{eq:k-range}; otherwise, we use the
fixed-dimensional integer-programming procedure above.
\end{remark}

\begingroup
\small
\setlength{\medskipamount}{3pt}
\setlength{\abovedisplayskip}{3pt}
\setlength{\belowdisplayskip}{3pt}
\setlength{\abovedisplayshortskip}{2pt}
\setlength{\belowdisplayshortskip}{2pt}
{\color{black}
\begin{algorithm}
\label{alg:jointct-kernel}
\leavevmode

\noindent\textbf{\em Input.}
A full-dimensional rational shifted simplicial cone $K\subseteq\mathbb R^d$.

\medskip
\noindent\textbf{\em Output.}
A finite list \(\mathcal O_{JCT}\) of  Laurent rational
functions such that
\begin{equation*}
  \sigma_{K} (\y) =  \sum_{i }  R_i(\y) =\sum_i \frac{c_i\,\y^{\alpha_{i,0}}}{\prod_{\ell =1}^r (1- \y^{\alpha_{i,\ell}}) },\qquad c_i\in\{-1,1\}.
\end{equation*}

\medskip
\begin{enumerate}[label=\textbf{Step \arabic*.},leftmargin=*,topsep=.4ex,partopsep=0pt,parsep=0pt,itemsep=.25ex]

\item Initialize
$\sigma_{K} (\y)$ with the constant-term formula as defined in
Proposition \ref{prop:simplicial-cone-ct}, and
then apply Proposition \ref{prop:focused-reduction}
(focused remainder reduction) to obtain the normalized form in \eqref{eq:normal-term}.
Initialize \(\mathcal O_{JCT}=\mathcal Q=\varnothing\).
If the initial normalized term has no binding variable, append it to
\(\mathcal O_{JCT}\) and return this list; otherwise, insert it into
\(\mathcal Q\).

\item Remove one term from \(\mathcal Q\) and write it as
\[
\Psi^\Lambda_{U,S} G(\Lambda):= \Psi^\Lambda_{U,S}
\left(
\frac{c\,u_0\Lambda^{c_0}}
{\prod_{\ell=1}^r
\bigl(1-u_{h+\ell}\Lambda^{c_\ell}\bigr)}
\right).
\]
If this expression is not normalized, we update it using Proposition~\ref{prop:focused-reduction} (focused remainder reduction).

\item
Focus on the last  coordinate $\lambda_h$ and set
$\gamma_\ell:=(c_\ell)_h$.
Use Lemma \ref{lem:minkowski} and Remark \ref{rem:LLLm}  to compute a multiplier \(k\).
Set
$
   g:=\gcd(k,s_h),
$
$
   b:=\frac{s_h}{g}.
$
Compute the unit lift \(\kappa\) of Lemma~\ref{lem:unit-lift} and integers \(e,t\) such that
\[\qquad
   1\le \kappa,e < s_h,
   \quad
   \kappa\equiv\frac{k}{g}\pmod b,
   \quad
   \gcd(\kappa,s_h)=1,e\kappa=1+ts_h.
\]

\item
Apply the split formula (Lemma \ref{lem:split}) to:
$$
   \Phi^\lambda_{u,s_h}G(\Lambda)
   =
   \Phi^{\nv_h}_{u^e,g}
   \Phi^{\nv_{h+1}}_{\nv_h^{-1},b}
   G\bigl(\Lambda_{h-1},u^t\nv_{h+1}^{\kappa}\bigr),
$$
as defined in \eqref{eq:focused-split-recalled}.
Perform the focused remainder reduction of $\Phi^{\nv_{h+1}}_{\nv_h^{-1},b}$ and the pole exchange identity of Proposition \ref{prop:pole-exchange} gives
\[
\Phi_{u_1,s_1}^{\lambda_1}\cdots
 \Phi_{u_h,s_h}^{\lambda_h}G(\Lambda)  =  \sum_{p_j >0}  \mathcal C_j,
\]
where
$$
 \mathcal C_j
 :={}
 \Phi_{u_1,s_1}^{\nv_1}\cdots
 \Phi_{u_{h-1},s_{h-1}}^{\nv_{h-1}}
 \Phi_{u^e,g}^{\nv_h}
 \Phi_{
  \widehat u_{h+j}
  \V_{h-1}^{\mathbf{\eta}_j}\nv_h^{\tau_j},
  p_j
 }^{\nv_{h+1}}
 H_j(\V).
$$
Detailed steps have been described in Section 5.1. % 无法识别标号

\item
Normalize each nonzero child in $\mathcal C_j$ separately.  Relabel the raw
binding variables as
\[
   \V=(\nv_1,\ldots,\nv_{h+1})
   :=
   (\lambda_1,\ldots,\lambda_{h-1},\nv_h,\nv_{h+1}).
\]
Theorem \ref{thm:smith-normalization} and Proposition \ref{prop:child-joint-realization} then give
\[
   \mathcal C_j
   =
   \Psi^\Lambda_{U_j,S_j}
   H_j\bigl(\V_{P_j}(\Lambda)\bigr).
\]
Perform the focused remainder reduction of $\Psi^\Lambda_{U_j,S_j}$.
If $\Lambda$ is empty, append the resulting
Laurent rational function to
\(\mathcal O_{\mathrm{JCT}}\).  Otherwise, append the resulting normalized
state to \(\mathcal Q\).

\item
If \(\mathcal Q\neq\varnothing\), return to \textbf{Step 2}.  If
\(\mathcal Q=\varnothing\), return the complete list
\(\mathcal O_{\mathrm{JCT}}\).
\end{enumerate}
\end{algorithm}

}
\endgroup

\section{Well-definedness}
\label{sec:jointct-invariants}

We show that every normalized term has at most \(d\) binding
variables and that a short multiplier can be constructed at every
nonterminal step. We also show that the cofactor has no poles at
the selected roots and that the binomials used in pole exchange
are pairwise coprime.

For a normalized state or a leaf \(T\) of the form
$$\Psi^\Lambda_{U,S}
   \left(
      \frac{c\,u_0\Lambda^{c_0}}
           {
            \prod_{\ell=1}^r
            \bigl(1-u_{h+\ell}\Lambda^{c_\ell}\bigr)}
   \right),
$$
write
\[
S=\diag(s_1,\ldots,s_h), \qquad
 C:=(\mathbf c_1,\ldots,\mathbf c_r)
\]
for its carried exponent matrix and its remaining binding-exponent matrix,
respectively.

\begin{theorem}
\label{thm:generation-preserved}
Every initial normalized state obtained from
Proposition~\ref{prop:simpcone-smith} or
Proposition~\ref{prop:simplicial-cone-ct}, every raw child, and every
normalized descendant satisfies
\begin{equation}\label{eq:generation}
 S\ZZ^h+C\ZZ^r=\ZZ^h.
\end{equation}
For a raw child, \(S\) is replaced by its carried exponent matrix \(B\).
The condition is preserved by  focused remainder reduction (Prop. \ref{prop:focused-reduction}),   pole exchange (Prop. \ref{prop:pole-exchange}),
the  split formula (Lemma \ref{lem:split}),
 and Smith
normalization (Theo. \ref{thm:smith-normalization}).
\end{theorem}

\begin{proof}
For an integer matrix \(N\in\ZZ^{q\times p}\) of full row rank, it is known that
\begin{equation}\label{eq:generation-determinantal}
 N\ZZ^p=\ZZ^q
 \quad\Longleftrightarrow\quad
 \delta_q(N)=1.
\end{equation}
Applied to \(N=[\,S\ C\,]\), criterion
\eqref{eq:generation-determinantal} is exactly
\(\delta_h([\,S\ C\,])=1\), which is equivalent to
\eqref{eq:generation}.

The initial cases follow from Remark~\ref{rem:initial-generation}.

Suppose that a Smith entry equal to \(1\) has been moved to the
first position.  If the remaining matrices \(S\) and \(C\) have \(h\) rows,
write
\[
 N=
 \begin{pmatrix}
  1&0&\boldsymbol\rho^{\mathsf T}\\
  0&S&C
 \end{pmatrix},
 \qquad
 \boldsymbol\rho\in\ZZ^r.
\]
Using the first column to eliminate \(\boldsymbol\rho^{\mathsf T}\) by
unimodular
column operations:
\[
 N\sim
 \begin{pmatrix}
  1&0&0\\
  0&S&C
 \end{pmatrix}.
\]
Hence
\[
 \delta_{h+1}(N)=\delta_h([\,S\ C\,]).
\]
Thus deleting a Smith entry equal to \(1\) also preserves the generation
condition.

Focused remainder reduction consists of column elementary operations, inversion of a Laurent
binomial changes the sign of one column, pole exchange permutes columns,
and Smith normalization consists of unimodular row and column operations.

Hence these operations preserve the criterion  \eqref{eq:generation-determinantal}.

Assume inductively that a normalized parent term satisfies the generation
condition.  Separate its focused coordinate and write
\[
 S=
 \begin{pmatrix}
  S'&0\\
  0&s
 \end{pmatrix},
 \qquad
 C=
 \begin{pmatrix}
  C'\\
  \boldsymbol\gamma
 \end{pmatrix},
 \qquad
 \boldsymbol\gamma=(\gamma_1,\ldots,\gamma_r).
\]
Thus, for
\[
 M:=[\,S\ C\,]
 =
 \begin{pmatrix}
  S'&0&C'\\
  0&s&\boldsymbol\gamma
 \end{pmatrix},
\]
the induction hypothesis is
\begin{equation*}
 \delta_h(M)=1.
\end{equation*}

We first treat the exact split.  Write \(s=bg\), and let
\(\kappa\) be the unit lift used in Lemma~\ref{lem:unit-lift}, so that
\[
 \gcd(\kappa,s)=1.
\]
With the rows ordered as
\[
 (\lambda_1,\ldots,\lambda_{h-1},\nv_h,\nv_{h+1}),
\]
and with the two new pole columns in the displayed order, the combined
pole-and-denominator matrix after the split is
\[
 M_{\rm sp}:=[\,S_{\rm sp}\ C_{\rm sp}\,]
 =
 \begin{pmatrix}
  S'&0&0&C'\\
  0&g&-1&0\\
  0&0&b&\kappa\boldsymbol\gamma
 \end{pmatrix}.
\]
Applying unimodular elementary operations to the last two rows and the two
new pole columns gives
\[
 M_{\rm sp}
 \sim
 \begin{pmatrix}
  S'&0&0&C'\\
  0&s&0&\kappa\boldsymbol\gamma\\
  0&0&1&0
 \end{pmatrix}.
\]
Therefore,
\begin{equation}\label{eq:split-determinantal-reduction}
 \delta_{h+1}(M_{\rm sp})=\delta_h(M_\kappa), \quad \text{where} \quad
 M_\kappa:=
 \begin{pmatrix}
  S'&0&C'\\
  0&s&\kappa\boldsymbol\gamma
 \end{pmatrix}.
\end{equation}
We prove by contradiction that \(\delta_h(M_\kappa)=1\).
Suppose, to the contrary, that there is a prime \(p\) dividing every
\(h\)-minor \(T_\kappa\) of \(M_\kappa\).
Let \(T\) denote the corresponding minor of \(M\).
Since \(\det S=s_1\cdots s_h\) is an \(h\)-minor and \(s_i\mid s_h=s\) for all \(i\),
we have \(p\mid s\).
It follows from \(\gcd(\kappa,s)=1\) that \(p\nmid \kappa\).

Now consider two cases according to whether
the displayed \(s\) belongs to \(T_\kappa\) or not.
If \(T_\kappa\) contains \(s\), then \(T=T_\kappa\) and hence \(p\mid T\).
If \(T_\kappa\) does not contain \(s\), then \(T_\kappa=\kappa T\).
Therefore \(p\mid \kappa T\), and since \(p\nmid \kappa\), we obtain \(p\mid T\).

Thus \(p\) divides every \(h\)-minor \(T\) of \(M\), contradicting
\(\delta_h(M)=1\).

By \eqref{eq:split-determinantal-reduction}, the exact split preserves the
generation condition. 
  The result follows by induction.
\end{proof}

The generation condition immediately bounds the number of carried
equations.

\begin{lemma}[Number of carried equations]\label{lem:h-le-r}
Let \(T\) be a nonterminal normalized state of the form
\eqref{eq:normal-term} occurring in Algorithm \ref{alg:jointct-kernel}.  Its
carried exponent matrix and remaining binding-exponent matrix are
\[
 S=\diag(s_1,\ldots,s_h)\in\ZZ^{h\times h},
 \qquad
 C=(\mathbf c_1,\ldots,\mathbf c_r)\in\ZZ^{h\times r},
\]
respectively, where \(1<s_1\mid\cdots\mid s_h\).  Thus \(h\) is the
number of carried equations and \(r\) is the number of remaining
denominator factors. Let \(d=\dim\mathcal K\) for the dimension
of input simplicial cone in Algorithm \ref{alg:jointct-kernel}.  Then
\begin{equation*}
 1\leq h\leq r\leq d.
\end{equation*}
Moreover, each row of \([\,S\ C\,]\) is primitive, meaning that the
greatest common divisor of its entries is \(1\).  In particular, if
\(s:=s_h\) and \(\gamma_\ell:=(\mathbf c_\ell)_h\), then
\begin{equation}\label{eq:focused-primitivity}
 \gcd(s,\gamma_1,\ldots,\gamma_r)=1,
\end{equation}
and hence \(\gamma_j\not\equiv0\pmod{s}\) for some \(1\leq j\leq r\).
\end{lemma}

\begin{proof}
Since a state with no carried equation is a leaf, every nonterminal state
has \(h\geq1\).  Initially \(r=d\), and no recursive operation increases
the number of remaining denominator factors, so \(r\leq d\).

By Theorem~\ref{thm:generation-preserved},
\[
 S\ZZ^h+C\ZZ^r=\ZZ^h.
\]
If \(h>r\), every \(h\times h\) minor of \([\,S\ C\,]\) contains a
column of \(S\).  Since \(S=\diag(s_1,\ldots,s_h)\) and
\(s_1\mid s_i\) for every \(i\), that column, and hence the minor, is
divisible by \(s_1>1\).  This contradicts
\eqref{eq:generation-determinantal}.  Therefore \(h\leq r\).

Since \(\delta_h([\,S\ C\,])=1\), each row of \([\,S\ C\,]\) is
primitive.  Its last row gives \eqref{eq:focused-primitivity}, and the
final assertion follows from \(s>1\).
\end{proof}

A raw child in
Step 4 of Alg. \ref{alg:jointct-kernel}
produced from an \(h\)-equation normalized parent has
\begin{equation*}
 h_{\mathrm{raw}}=h+1\leq r+1
\end{equation*}
carried equations before Smith normalization, and one step
 (Step 4 of Alg. \ref{alg:jointct-kernel})
produces at
most \(r\) such raw children.  The inequality \(h\leq r\) is therefore a property of
Smith normalized states.

The next invariant ensures both the regularity of every distributed
summand and the pairwise coprimality required for each pole exchange.

To state it, write
\[
 u_\ell=\y^{\boldsymbol\alpha_\ell}
 \qquad(1\leq\ell\leq h+r)
\]
in the exponent lattice of the current marker variables.  The carried and
remaining full exponent columns of a recursive term form the matrix below,
where \(B\) is its carried exponent matrix and
\(C=(\mathbf c_1,\ldots,\mathbf c_r)\):
\begin{equation}\label{eq:full-exponent-matrix}
 \widehat M
 =
 \begin{pmatrix}
  (\boldsymbol\alpha_1,\ldots,\boldsymbol\alpha_h)&
  (\boldsymbol\alpha_{h+1},\ldots,\boldsymbol\alpha_{h+r})\\
  B&C
 \end{pmatrix}.
\end{equation}

\begin{lemma}[Full-column independence, regularity, and coprimality]
\label{lem:regularity-coprimality}
For both initial normalized \texttt{SimpCone[S]} states and directly
supplied simplicial-cone terms, the columns of the full exponent matrix
\(\widehat M\) in \eqref{eq:full-exponent-matrix} are linearly independent
over \(\QQ\) and generate a saturated sublattice of the full exponent
lattice.  Both properties are preserved in every descendant.  Hence
each descendant satisfies: (i) its cofactor is defined at every common root
of its carried equations; and (ii) the Laurent binomials in each
one-variable pole exchange are pairwise coprime.
\end{lemma}

\begin{proof}
Two Laurent binomials considered here, each equal to \(1\) minus a Laurent
monomial, are coprime if and only if their nonzero full exponent columns
are linearly independent over \(\QQ\).

For an initial \texttt{SimpCone[S]} term, the marker block is a column
permutation of \(I_n\), and Smith postprocessing applies a unimodular
combination to its carried columns.  For a directly supplied simplicial
cone, the full exponent matrix in \eqref{eq:simplicial-cone-ct} is block triangular
with unimodular diagonal blocks \(Q\) and \(-P\).  Thus the initial full
columns are independent and generate a saturated sublattice in both cases.

The recursive operations preserve this independence.  Smith
normalization uses unimodular row operations and unimodular combinations
of carried columns.  Remainder reduction is a column shear, inversion
negates a column, and pole exchange permutes columns.  Deleting a unit
Smith entry removes, after column shears, an isolated pivot row together
with its column.  These integral operations preserve both properties.

For the exact split, reduce the full exponent matrices modulo an arbitrary
prime.  If \(a\) is the coefficient of the outer new pole column,
\(\mathbf z\) collects the coefficients of the remaining denominator
columns, and \(G=\boldsymbol\gamma\mathbf z\), then the two new binding rows
force the inner pole coefficient to be \(ga\) and give
\[
 as+\kappa G=0.
\]
The coefficient
\[
 x=ea+tG
\]
then gives a relation before the split, since
\[
 xs+G=e(as+\kappa G)=0.
\]
Conversely, a relation before the split with focused coefficient \(x\)
gives one after the split by taking \(a=\kappa x\) and inner coefficient
\(g\kappa x\).  The identity \(e\kappa=1+ts\) shows that these
correspondences are inverse.  Thus the full columns have full rank modulo
every prime before the split if and only if they do so afterwards.  Since
an independent set of integer columns spans a saturated sublattice exactly
when it has full column rank modulo every prime, saturation and independence
hold in every descendant.

For regularity, suppose that a remaining denominator monomial were \(1\)
at a common root of the carried equations.  Raising this equality to
\(\abs{\det B}\) and eliminating the binding variables with the carried
equations gives an equality of marker monomials.  Since the markers are
algebraically independent, the corresponding remaining full exponent
column is then a rational linear combination of the carried columns,
contradicting full-column independence.  Thus every descendant cofactor is
defined on the common root set.

For coprimality, any two columns of \(\widehat M\) are independent, so the
criterion above shows that the Laurent binomials in each pole exchange are
pairwise coprime.
\end{proof}

Lemma~\ref{lem:regularity-coprimality} supplies the hypothesis that every
summand is defined separately; cancellation of poles between different
summands is not used.  When a complete denominator monomial is viewed as a
Laurent monomial in \(\nv_{h+1}\), its coefficient is independent of
\(\nv_{h+1}\); every \(\nv_{h+1}\)-power belongs in the displayed pole order.  The
lemma justifies distributing the outer average over a finite sum.  It does
not assert that the $\nv_h$- and $\nv_{h+1}$-averages commute.

\section{Complexity}
\label{sec:complexity}

Throughout this section, monomials are stored by their integer exponent
vectors.  Smith normal forms and their transformation matrices are
computed by the polynomial-time algorithm of \cite{KannanBachem1979},
and short multipliers by the fixed-dimensional integer-programming
procedure in Remark~\ref{rem:LLLm}.

\begin{lemma}[Bit length of recursive states]
\label{lem:recursive-bit-length}
For a normalized state \(T\) of Algorithm~\ref{alg:jointct-kernel},
including a marker-only leaf, write
\(u_i=\y^{\boldsymbol\alpha_i}\), with the scalar sign recorded
separately, and let \(\ell(T)\) be the maximum binary length of the
entries of \(\boldsymbol\alpha_0,\ldots,\boldsymbol\alpha_{h+r}\).
For fixed \(d\), there is a nondecreasing polynomial \(p_d\) such that
every normalized child \(T_j\) of a nonterminal state \(T\) with
determinant \(D=s_1\cdots s_h\) satisfies
\begin{equation}\label{eq:recursive-bit-length}
 \ell(T_j)\leq\ell(T)+p_d\bigl(\log(2+D)\bigr).
\end{equation}
\end{lemma}

\begin{proof}
By \eqref{eq:normalized-centered-range}, all binding exponents have
absolute value at most \(D\).  Since \(1\leq\kappa,e<s\leq D\) and
\(0\leq t<s\), the split, inner remainder reduction, and pole exchange
produce raw binding data, including the numerator exponents, of magnitude
\(O_d(D^2)\).  The raw marker exponents are integer linear combinations
of the parent marker exponents with coefficients of magnitude \(O_d(D^2)\).
These coefficient bounds depend only on the binding data.

The matrix \(B_j^{\rm raw}\) has order at most \(d+1\) and binary
encoding length \(O_d(\log(2+D))\).  Hence the entries of its Smith
transformation matrices \(P_j,Q_j\) have binary length bounded by a
polynomial in \(\log(2+D)\).  The transformed binding exponents and the
quotients used in the final remainder reductions satisfy the same type
of bound.  Smith normalization, deletion of unit entries, and remainder
reduction therefore express each child marker exponent vector as an
integer linear combination of the parent vectors with coefficients of
polynomially bounded binary length in \(\log(2+D)\).  The monomials
absorbed into the numerator by binomial inversion obey the same bound.
Since only \(O_d(1)\) vectors occur, taking binary lengths proves
\eqref{eq:recursive-bit-length}, after enlarging \(p_d\) if necessary.
\end{proof}

\begin{theorem}[Complexity of Algorithm \ref{alg:jointct-kernel}]
\label{thm:complexity}
Let \(\mathcal K\subseteq\mathbb R^d\) be a full-dimensional rational
shifted simplicial cone, with fixed \(d\geq1\), and let \(D_0\) be the
determinant of the normalized state obtained in
Step~1 of Algorithm~\ref{alg:jointct-kernel}.  Then the algorithm has
recursion depth
\[
 O_d\bigl(1+\log\log(2+D_0)\bigr)
\]
and expresses \(\sigma_{\mathcal K}(\y)\) as a signed sum of at most
\[
 N_{\mathrm{out}}
 \leq
 (1+\log D_0)^{O_d(1)}
\]
generating functions of possibly shifted unimodular simplicial cones.
Its bit complexity is polynomial in the binary input length, including
the shift.
\end{theorem}

\begin{proof}
Consider a nonterminal state with \(h\) carried equations, \(r\) remaining
denominator factors, determinant \(D=s_1\cdots s_h\), and largest Smith
entry \(s=s_h\).  By Lemma~\ref{lem:h-le-r}, \(h\leq r\leq d\).  If
\(D_j\) is the determinant of a nonzero child after normalization, then
\eqref{eq:raw-child-determinant} and \eqref{eq:child-focused-order} give
\[
\begin{aligned}
 D_j
 &=\frac{D}{s}\,g p_j
 =\frac{D}{s}\,g[\kappa\gamma_j]_b=\frac{D}{s}[k\gamma_j]_s.
\end{aligned}
\]
Here \(g p_j=[k\gamma_j]_s\) follows from \(s=gb\) and
\(\kappa\equiv k/g\pmod b\).
Since \(s=s_h\) is the largest Smith entry, \(s_i\leq s\) for every
\(i\), and hence
\[
 D=s_1\cdots s_h\leq s^h,
 \quad
 s\geq D^{1/h}.
\]
On the other hand, \eqref{eq:k-range} gives
\([k\gamma_j]_s\leq s^{(r-1)/r}\).  Substitution into the preceding
formula for \(D_j\) yields
\[
 D_j
 \leq\frac{D}{s}s^{(r-1)/r}
 =\frac{D}{s^{1/r}}.
\]
Now \(s^{1/r}\geq D^{1/(hr)}\), so
\[
 D_j
 \leq D^{1-1/(hr)}
 \leq D^{1-1/d^2},
\]
where the last inequality follows from \(h\leq r\leq d\), and hence
\(hr\leq d^2\).

The centered residue also satisfies \([k\gamma_j]_s\leq s/2\), so
\[
 D_j\leq\left\lfloor\frac D2\right\rfloor.
\]
Thus the effective bound is
\[
 D_j
 \leq
 \min\!\left\{
 D^{1-1/d^2},
 \left\lfloor\frac D2\right\rfloor
 \right\}.
\]
The halving bound is the sharper one when \(D<2^{d^2}\).

If \(D=1\), all unit Smith entries have been deleted, so Step~1 or
Step~5 outputs a marker-only leaf and creates no further child.

Assume \(d\geq2\) and put \(\theta_d=1-d^{-2}\).  Starting from \(D_0\),
while the current determinant is at least \(2^{d^2}\), a branch satisfies
\[
 D_q\leq D_0^{\theta_d^q}
\]
after \(q\) levels.  If \(D_0<2^{d^2}\), set \(q_0=0\).  Otherwise, set
\[
 q_0
 :=
 1+\left\lfloor
 \frac{
  \log\!\left(\log D_0/(d^2\log 2)\right)
 }{
  -\log(1-d^{-2})
 }
 \right\rfloor.
\]
Then
\[
 \theta_d^{q_0}\log D_0<d^2\log 2,
 \qquad
 D_{q_0}
 \leq\exp\!\left(\theta_d^{q_0}\log D_0\right)
 <2^{d^2}.
\]
Since \(-\log(1-d^{-2})\) depends only on \(d\), this gives
\[
 q_0=O_d\bigl(1+\log\log(2+D_0)\bigr).
\]
Thereafter, repeated use of the halving bound gives
\[
 D_{q_0+t}
 \leq
 \left\lfloor\frac{D_{q_0}}{2^t}\right\rfloor.
\]
Because the determinants are positive integers and
\(D_{q_0}<2^{d^2}\), the branch reaches \(D=1\) within at most
\(d^2-1\) further levels.  Thus its depth is at most \(q_0+d^2-1\).
When \(d=1\), every child already has determinant \(1\).  This proves the
asserted depth bound.

Every node has at most \(r\leq d\) children.  Therefore
\[
 N_{\mathrm{out}}
 \leq
 d^{O_d(1+\log\log(2+D_0))}
 =
 (1+\log D_0)^{O_d(1)}.
\]
Finally, the exact split, pole exchange, joint realization, and Smith
normalization are identities, so the signed sum of the leaves is
\(\sigma_{\mathcal K}(\y)\).
Since the recursion preserves the number \(d\) of remaining denominator
factors, Lemma~\ref{lem:regularity-coprimality} shows that their exponent
columns at a leaf (\(h=0\)) form a basis of \(\ZZ^d\).
Hence each leaf is, up to its scalar sign, the generating function of an
integer translate of a unimodular simplicial cone.

Let \(T_0\) be the initial normalized state and \(L\) the maximum
recursion depth.  Since \(D\leq D_0\) at every descendant,
Lemma~\ref{lem:recursive-bit-length} gives
\[
 \ell(T)\leq\ell(T_0)+L\,p_d\bigl(\log(2+D_0)\bigr)
\]
for every state \(T\).  Initialization has polynomial bit complexity,
and both \(\ell(T_0)\) and \(\log(2+D_0)\) are polynomially bounded in
the binary input length.  Each normalized state contains only
\(O_d(1)\) integer entries, with binding exponents bounded by \(D_0\);
the preceding lemma also bounds the intermediate data in a recursive
step.  Thus each node is processed in polynomial time.  The depth and
branching bounds give polynomially many nodes, proving the asserted
bit complexity.
\end{proof}

\begin{remark}[Comparison with Barvinok's algorithm]
The conclusion above is parallel to the fixed-dimensional cone
decomposition underlying Barvinok's algorithm
\cite{Barvinok1994,BarvinokPommersheim1999}: both methods express the
generating function of a simplicial cone as a polynomial-size signed sum
of unimodular generating functions.  Writing \(D_0\) for the respective
initial progress parameter, both have
\(O_d(1+\log\log(2+D_0))\) recursion depth and polynomially many leaves in
\(1+\log D_0\).  For \(d\geq2\), Barvinok's classical short-vector step gives
the stronger uniform index bound
\(D_{\mathrm{child}}\leq D^{(d-1)/d}\),
whereas the joint recursion above combines the uniform bound
\(D_j\leq D^{1-1/d^2}\) with the complementary halving bound.
Their recursive mechanisms are also different.  Barvinok's algorithm
constructs a geometric signed decomposition into cones of smaller index,
whereas Algorithm~\ref{alg:jointct-kernel} reduces the carried determinant
algebraically through exact splitting, pole exchange, joint root averages,
and Smith normalization.  Its intermediate children are therefore
root-average terms rather than cones in a geometric subdivision.
\end{remark}

\section{Concluding remarks}
\label{sec:conclusion}

We have presented a polynomial-time algorithm for lattice-point counting in fixed dimension that operates entirely within the framework of nested root averages and does not rely on Barvinok's unimodular decomposition. The algorithm takes as input a full-dimensional rational shifted simplicial cone and computes its lattice-point generating function via a recursive pole-exchange procedure. The key technical ingredients are a residue-lattice argument, based on Minkowski's theorem, that produces a short multiplier for the non-coprime split, and the use of Smith normal form to encode the coupled root equations of each child as a joint root average. Two structural invariants---the generation condition and full-column independence---ensure that the recursion is well defined and that all pole exchanges are valid.

For a fixed-dimensional simplicial cone, the algorithm achieves recursion depth \(O_d(1+\log\log(2+\ind(\mathcal K^*)))\) and produces a signed sum of at most \((1+\log \ind(\mathcal K^*))^{O_d(1)}\) unimodular cone generating functions. To our knowledge, this is the first polynomial-time fixed-dimensional algorithm for lattice-point counting that entirely avoids Barvinok decomposition. Moreover, the root-averaging framework naturally accommodates
explicitly given Laurent polynomial numerators, which arise
frequently in MacMahon's partition analysis.

Several directions for future work remain. First, while the complexity bounds are polynomial in the input size for fixed \(d\), the exponents depend on \(d\) and are not optimized. A refined analysis may yield sharper constants. Second, a practical implementation of the algorithm would allow direct comparison with \texttt{LattE} and the \texttt{barvinok} library on benchmark families of polytopes. Finally, the residue-lattice approach can be adapted to parametric counting problems, like the examples presented in \cite{XinXuZhang2025}.

We believe that the joint root-average framework introduced here opens a promising alternative direction in the algorithmic theory of lattice-point enumeration, complementary to the well-established Barvinok paradigm.

\noindent
{\small \textbf{Acknowledgments:}}
%The authors would like to thank the anonymous referees for their valuable suggestions for improving the presentation.
%The authors would like to express sincere gratitude for all the suggestions that have improved the presentation of this paper.
The authors would like to thank Dr. Sihao Tao for helpful discussions on this work.
Guoce Xin is partially supported by the National Natural Science Foundation of China (12571355).

\end{document}